\documentclass[11pt,reqno]{amsart}
\usepackage{amsmath,amsxtra,latexsym,amsthm,amssymb,amscd,pb-diagram}
\usepackage{mathrsfs,mathabx,amsfonts}
\usepackage[colorlinks=true,linkcolor=magenta,citecolor=blue]{hyperref}
\usepackage[margin=1in]{geometry}

\usepackage{bm}
\usepackage{color}
\usepackage{makecell,multirow,diagbox}
\usepackage{mathtools}
\usepackage[displaymath, mathlines]{lineno}

\usepackage{graphicx}
\usepackage{epsfig}
\usepackage{array}
\usepackage{enumitem}

\newtheorem{theorem}{Theorem}[section]
\newtheorem{proposition}{Proposition}[section]
\newtheorem{lemma}{Lemma}[section]

\newtheorem{remark}{Remark}[section]

\newcommand{\R}{\mathbb{R}}

\newcommand{\ity}{\infty}

\newcommand{\f}{\displaystyle\frac}

\begin{document}
\title[Asymptotic profiles for $\sigma$-evolution equations with double damping]{Sharp higher-order $L^2$-asymptotic expansion of solutions to $\sigma$-evolution equations with different damping types}

\thanks{$^* $\textit{Corresponding author:} Tuan Anh Dao (anh.daotuan@hust.edu.vn)}

\maketitle
\centerline{\scshape Dinh Van Duong$^{1}$, Tuan Anh Dao$^{1, *}$}
\medskip
{\footnotesize
	\centerline{$^1$ Faculty of Mathematics and Informatics, Hanoi University of Science and Technology}
	\centerline{No.1 Dai Co Viet road, Hanoi, Vietnam}}
\medskip

\begin{abstract}
   In this paper, our main goal is to achieve the sharp higher-order asymptotic expansion of solutions to $\sigma$-evolution equations with different damping types in the $L^2$ framework by effectively using Taylor series expansion combined with Fa\`a di Bruno's formula for higher-order derivatives of certain specific functions. In particular, we observe the influence of ``parabolic-like models'' corresponding to $\sigma_1 \in [0, \sigma/2)$ and ``$\sigma$-evolution like models'' corresponding to $\sigma_2 \in (\sigma/2, \sigma]$ on the higher-order asymptotic behavior of solutions. 
\end{abstract}

% \linenumbers
\tableofcontents

%====================================================================================
%=================================================================================={Introduction}	
\section{Introduction}
Let us consider the following Cauchy problem for linear $\sigma$-evolution equations with double damping:
\begin{equation} \label{Main.Eq.1}
\begin{cases}
u_{tt}+ (-\Delta)^\sigma u+ (-\Delta)^{\sigma_1} u_t+ (-\Delta)^{\sigma_2} u_t= 0, &\quad x\in \R^n,\, t > 0, \\
u(0,x)= u_0(x),\quad u_t(0,x)= u_1(x), &\quad x\in \mathbb{R}^n, \\
\end{cases}
\end{equation}
where $\sigma \geq 1$ is assumed to be any fractional number and $0\leq \sigma_1< \sigma/2< \sigma_2 \leq \sigma$. The operator $(-\Delta)^{\sigma}$ stands for the fractional Laplacian with the $\sigma$-th power, well-known as high-order waves, where the parameter $\sigma$ may be non-integer. In the non-integer case, it is determined by
$$(-\Delta)^{\sigma} \psi (x) = \mathfrak{F}^{-1}_{\xi \to x}\left(|\xi|^{2\sigma} \widehat{\psi}(\xi)\right)(x), $$
where $\psi \in H^{2\sigma}(\mathbb{R}^n)$  and $\widehat{\psi}$ denotes the Fourier transform of a function $\psi$. The term $(-\Delta)^{\theta} u_t$ for some $\theta \geq 0$ is called a damping term, specifically, it is often referred to the frictional (or external) damping, the parabolic like damping, the $\sigma$-evolution like damping and the visco-elastic (or strong) damping when $\theta = 0$, \,$\theta = \sigma_1 \in (0, \sigma/2)$,\, $\theta = \sigma_2 \in (\sigma/2, \sigma)$ and $\theta = \sigma$, respectively. \medskip

%\textcolor{red}%{\begin{itemize}
  %  \item Đầu tiên, em nhắc lại 2 kiểu mô hình xuất hiện trong PT (1) bao gồm ``parabolic'' khi $\sigma_1 \in [0,\sigma/2)$ và $\mu_2=0$, ``$\sigma$-evolution'' khi $\sigma_2\in (\sigma/2,\sigma]$ và $\mu_1=0$ là gì?
 %   \item Em giới thiệu các kết quả cho PT (1) với $\sigma_1=0$ và $\sigma_2=\sigma$. Bắt đầu từ trường hợp $\sigma=1$, họ đã làm được những gì (Ikehata, Swanda, Takeda, Reissig, Kainanne, ... ). Sau đó, trường hợp tổng quát $\sigma\ge 1$, họ đã làm được gì. Đặc biệt, họ đã có kết quả về khai triển dáng điệu tiệm cận bậc cao của nghiệm chưa (Dao, Michihisa, Reissig, Kainane, ... )?
  %  \item Tiếp đến, phân tích về các kết quả cho PT (1) với $\sigma_1 \in [0,\sigma/2)$ và $\sigma_2\in (\sigma/2,\sigma]$ (chính là bài báo của 3 thầy trò) đã có kết quả là gì? Nhấn mạnh về dáng diệu tiệm cận bậc 1 thì chỉ được mô tả bởi tham số $\sigma_1$ không thấy sự ảnh hưởng của $\sigma_2$ mà chỉ thấy được $\sigma_2$ tác động đến tính chính quy. Sau đó, xoáy sâu vào cái mới của bài báo này là chỉ rõ được sự ảnh hưởng thực sự của $\sigma_2$ ngoài tính chính quy để mô tả dáng điệu tiệm cận của nghiệm như thế nào? Nó chính là khai triển tiệm cận bậc cao thì mới thấy được điều này.
%\end{itemize}}
To get started, we now mention some historical views in terms of studying the following problem, which becomes the main inspiration to consider \eqref{Main.Eq.1}:
\begin{align}
\begin{cases}
    u_{tt} +(-\Delta )^{\sigma} u+ (-\Delta)^{\theta} u_t = 0, &x \in \mathbb{R}^n,\, t > 0, \\
    u(0,x) = u_0(x), \quad u_t(0,x) = u_1(x), &x \in \mathbb{R}^n.
    \end{cases}\label{Eq.2}
\end{align}
where $\theta \in [0, \sigma]$. The most famous case of this problem is when $\sigma = 1$, which has been studied by many authors for a long time (for example, \cite{Matsumura1976, Michihisa2021, Shibata2000, DabbiccoReissig2014, Takeda2015, Ikehata2021, KawakamiTakeda2016}). Matsumura \cite{Matsumura1976} was the first to establish some basic decay estimates for the problem (\ref{Eq.2}) with $(\sigma, \theta) = (1, 0)$, also well-known as the damped wave equation, and concluded that the damped wave equation has a diffusive structure as $t \to \infty$. In recent years, many mathematicians have focused on studying a typical nonlinear problem of (\ref{Eq.2}) for $(\sigma, \theta) = (1, 0)$ with the right-hand side replaced by the nonlinear term $|u|^p$. Todorova-Yordanov \cite{TodorovaYordanov2001} verified the critical exponent of this problem, which is $p_{\rm c}= 1+2/n$. Here, the critical exponent $p_{\rm c}$ is understood as follows: If the exponent $p$ of the nonlinear term  $|u|^p$ satisfies $p > p_{\rm c}$, then the problem has a unique small data global (in time) solution, moreover, if the exponent $p$ satisfies $p < p_{\rm c}$, then the problem does not have any global (in time) solution. Additionally, the higher-order asymptotic behavior of solutions in $L^q$ framework for some $q\in (1,\ity)$ has also been presented in the papers \cite{Michihisa2021, Takeda2015, KawakamiTakeda2016} in which they concluded the approximation of
solutions by the Gauss kernel. For $\sigma = 1$ and $\theta \in (0,1)$, there are a series of papers devoting to (\ref{Eq.2}) in this case, for instance \cite{DabbiccoReissig2014, IkehataTakeda2019, DAbbiccoEbert2014, NarazakiReissig2013}. Namely, D'Abbicco-Reissig in \cite{DabbiccoReissig2014} observed the differences in the solution properties between \textit{parabolic type models} corresponding to $\theta \in (0, 1/2)$ and \textit{hyperbolic type models} corresponding to $\theta \in (1/2, 1)$. Next, D'Abbicco-Ebert in \cite{DAbbiccoEbert2014} described the asymptotic behavior of solutions to (\ref{Eq.2}) with parabolic type models in the $L^m-L^q$ framework with $1 \leq m \leq q \leq \infty$. However, the used method in their proofs seems not applicable to hyperbolic type models which was investigated in \cite{IkehataTakeda2019} afterwards. More precisely, the authors in the cited papers clearly classified the asymptotic profiles of solutions to (\ref{Eq.2}) according to each specific model and they explicitly stated the asymptotic behavior of the higher-order derivatives of solutions. As a result, optimal decay estimates for solutions can be demonstrated from the viewpoint of the higher order derivatives in the $L^2$-sense. More generally, for $\sigma \geq 1$ and $\theta \in [0, \sigma]$, the problem (\ref{Eq.2}) has been widely concerned in a lot of papers \cite{DaoReissig2019_1, DaoReissig2019_2, DAbbiccoEbert2017, Ikehata2021, DAbbiccoEbert2022, DuongKainaneReissig2015, DAbbiccoEbert2021}. 
%Another important model of (\ref{Eq.2}) when $\sigma = 2$ is well-known as the beam equation or the plate equation in the case of space dimension $n= 1$ or $n=2$, respectively.
Among other things, Dao-Reissig in \cite{DaoReissig2019_1, DaoReissig2019_2} obtained $L^1$ estimates for solutions to (\ref{Eq.2}) by using the theory of modified Bessel functions combined with Fa\`a di Bruno's formula introduced in \cite{NarazakiReissig2013}. One recognizes that this approach is just very effective for ``parabolic like model'' corresponding to $\theta \in (0, \sigma/2)$ as well as ``$\sigma$-evolution-like model'' corresponding to $\theta \in (\sigma/2, \sigma)$. Unfortunately, it is not applicable to take into consideration the case of strongly damped $\sigma$-evolution equations, i.e. $\theta = \sigma$. Quite recently, D'Abbicco-Ebert in \cite{DAbbiccoEbert2021} obtained optimal $L^p-L^q$ estimates for solutions by treating separately the two different components of solutions including the oscillatory one and the diffusive one in a particular $t$-dependent zone of the frequency space. Ikehata in \cite{Ikehata2021} derived sharp decay/non-decay estimates of solutions to (\ref{Eq.2}) with $\theta = \sigma$ in the $L^2$-sense for any spatial dimension. \medskip

From these above-mentioned observations, we can say that the investigation of (\ref{Eq.2}) devoting only one damping term seems to be complete. To address the mix of two different damping terms, now we will discuss several recent results related to the problem (\ref{Main.Eq.1}). The authors in \cite{IkehataSawada2016, IkehataMichihisa2019} obtained the higher-order asymptotic behavior for solutions to (\ref{Main.Eq.1}) with $\sigma_1 = 0$ and $\sigma_2=\sigma=1$ under more heavy moment conditions on the initial data for any space dimension by applying Taylor expansion theorem effectively. More generally, in the case where $\sigma_1 = 0$ and $\sigma_2 = \sigma \geq 1$ are arbitrary, (\ref{Main.Eq.1}) is the so-called $\sigma$-evolution equations with frictional and visco-elastic (or strong) damping. Dao-Michihisa in \cite{DaoMichihisa2020} successfully found not only  higher order asymptotic expansions of solutions but also diffusion phenomenon in the $L^p-L^q$ framework, with $1 \leq p \leq q \leq \infty$. Among other things, they underlined that the presence of frictional damping affected the profiles of the solution at low frequencies, whereas their large frequency profile is modified by the presence of the visco-elastic damping. This present work is a continuation of the very recent paper of Dao-Duong-Nguyen \cite{DaoDuongNguyen2024} in which the authors have found the asymptotic behavior for the solution to problem (\ref{Main.Eq.1}) for any $\sigma \geq 1$ and $\sigma_1 \in [0, \sigma/2)$, $\sigma_2 \in (\sigma/2, \sigma]$, namely
\begin{align}
    u(t, x) \sim \left(\int_{\mathbb{R}^n} u_0(x) dx \right)\mathfrak{F}_{\xi \to x}^{-1}\left(e^{-|\xi|^{2(\sigma-\sigma_1)}t}\right) + \left(\int_{\mathbb{R}^n} u_1(x) dx \right) \mathfrak{F}_{\xi \to x}^{-1}\left(\frac{e^{-|\xi|^{2(\sigma-\sigma_1)}t}}{|\xi|^{2\sigma_1}}\right). \label{1_order_Asym}
\end{align}
Furthermore, they concluded about the essential interaction between ``parabolic-like model'' and ``$\sigma$-evolution-like models'' when they appear together in an equation. One realizes that the damping term representing ``parabolic like model'' has a greater influence in a comparison with the one standing for  
 ``$\sigma$ evolution like models'' in determining the first-order asymptotic behavior of solutions to (\ref{Main.Eq.1}). More precisely, we can see in \eqref{1_order_Asym} that this profile is described by the two Fourier multipliers only depending on the parameters $\sigma$ and $\sigma_1$. Meanwhile, the damping term corresponding to ``$\sigma$ evolution like models'' completely decides the necessary regularity for both the initial data and the solution to obtain these asymptotic behaviors. To the best of the authors' knowledge, it seems that so far there is no any paper dedicated to the study of the higher-order asymptotic profile of solutions to (\ref{Main.Eq.1}) for any $\sigma \geq 1$ and $\sigma_1 \in [0, \sigma/2)$, $\sigma_2 \in (\sigma/2, \sigma]$. For this reason, the main contribution of this paper is not only to explore such results in the $L^2$ framework but also to discover the role of the damping term representing ``$\sigma$-evolution like models'', i.e the parameter $\sigma_2$, on describing the asymptotic profile of solutions, which never appears in \cite{DaoDuongNguyen2024}. To concretize this point, we aim to establish the higher-order asymptotic behavior for solutions to (\ref{Main.Eq.1}) as follows:
\begin{align}
    u(t,x) \sim \mathfrak{F}^{-1}_{\xi \to x}\left(\mathcal{A}_0^k(t, \xi)\right) \ast_x u_0(x) + \mathfrak{F}^{-1}_{\xi \to x}\left(\mathcal{A}_1^k(t, \xi)\right) \ast_x u_1(x), \label{2_order_Asym}
\end{align}
where the quantities $\mathcal{A}_0^k (t,\xi)$ and $\mathcal{A}_1^k (t,\xi)$ stand for the $k$-th order asymptotic behavior of the solution, which will be introduced in Notation section. In addition, another novelty of this paper is that we have provided the sharp estimate for the error term in the higher-order asymptotic expansion of solutions for large time, that is,
\begin{align*}
    &\left\| u(t,x) - \mathfrak{F}^{-1}_{\xi \to x}\left(\mathcal{A}_0^k(t, \xi)\right) \ast_x u_0(x) - \mathfrak{F}^{-1}_{\xi \to x}\left(\mathcal{A}_1^k(t, \xi)\right) \ast_x u_1(x) \right\|_{L^2} \\
    &\hspace{3cm} \sim (1+t)^{-\frac{n}{4(\sigma-\sigma_1)}+\frac{\sigma_1}{\sigma-\sigma_1}-k\frac{\delta}{\sigma-\sigma_1}}.
\end{align*}

\textbf{The remaining part of this paper is organized as follows:} In Section \ref{Main_results.Sec}, we are going to state notations which will be used throughout this paper and main results indicating the higher-order asymptotic profile of solutions to (\ref{Main.Eq.1}) in the $L^2$ framework. Then, the proofs of main results will be presented more detail in Section \ref{section2}. %Eventually, we compare the results obtained in Section \ref{section2} with previous ones in some spatial dimensions in Section \ref{section3}.

\section{Notations and Main results} \label{Main_results.Sec}
\subsection{Notations}
To give some notations which will be used in the next sections, let us consider the following referenced Cauchy problem for (\ref{Main.Eq.1}):
\begin{align}
    \begin{cases}
        b u_{tt} +  (-\Delta)^{\sigma} u+ a_1(-\Delta)^{\sigma_1} u_t + a_2 (-\Delta)^{\sigma_2} u_t = 0, &x \in \mathbb{R}^n, \, t > 0, \\
        u(0,x) = u_0(x), \quad u_t(0,x) = u_1(x), &x \in \mathbb{R}^n,
    \end{cases} \label{Main.Eq.2}
\end{align}
 with two positive constants $a_1, a_2$ and a sufficiently small number $b>0$. At first glance, the problem (\ref{Main.Eq.2}) can be understood as a generalization in a comparison with (\ref{Main.Eq.1}). The main reason to concern the problem \eqref{Main.Eq.2} comes from the fact that the first-order asymptotic behavior for solutions to (\ref{Main.Eq.1}) is like parabolic models (when $a_2=0$) and has the property of $\sigma$-evolution like models (when $a_1=0$), which was sharply presented in \cite{DaoDuongNguyen2024}. More precisely, the solutions to (\ref{Main.Eq.1}) behave the same as those to the following anomalous diffusion problem:
\begin{align*}
\begin{cases}
    v_t + \nu (-\Delta)^{\sigma-\sigma_1} v = 0, & x \in \mathbb{R}^n, t > 0,\\
    v(0,x) = v_0(x), &x \in \mathbb{R}^n,
    \end{cases}
\end{align*}
where both a constant $\nu$ and the initial data $v_0$ are chosen suitably. Rough speaking, this problem is also as a consequence of \eqref{Main.Eq.2} by taking $b=0$ and $a_2=0$. For the purpose to describe the higher-order asymptotic behavior for solutions to (\ref{Main.Eq.1}), i.e \eqref{Main.Eq.2} when $(a_1,a_2,b)=(1,1,1)$, we can say that taking account of (\ref{Main.Eq.2}) plays an important role to show the interplay of all parameters $\sigma$, $\sigma_1$ and $\sigma_2$. To get started, we apply the Fourier transform to \eqref{Main.Eq.2} to have
\begin{align}
    \begin{cases}
        b \widehat{u}_{tt} + |\xi|^{2\sigma} \widehat{u}+ \big(a_1|\xi|^{2\sigma_1} + a_2 |\xi|^{2\sigma_2}\big) \widehat{u}_t = 0 &\xi \in \mathbb{R}^n, \, t > 0, \\
        \widehat{u}(0,\xi) = \widehat{u}_0(\xi), \quad \widehat{u}_t(0,\xi) = \widehat{u}_1(\xi) &\xi \in \mathbb{R}^n.
    \end{cases} \label{Main.Eq.2_Fourier}
\end{align} 
The solution to (\ref{Main.Eq.2_Fourier}) can be expressed in the form
\begin{align*}
    \widehat{u}(t,\xi, a_1, a_2, b) &= \frac{\lambda_1^0(|\xi|, a_1, a_2,b) e^{\lambda_2^0(|\xi|,a_1, a_2,b) t}-\lambda_2^0(|\xi|,a_1, a_2,b) e^{\lambda_1^0(|\xi|,a_1, a_2,b) t}}{\lambda_1^0(|\xi|, a_1, a_2,b) -\lambda_2^0(|\xi|, a_1, a_2,b)} \widehat{u_0}(\xi) \\
    &\qquad + \frac{e^{\lambda_1^0(|\xi|, a_1, a_2,b) t}- e^{\lambda_2^0(|\xi|,a_1, a_2,b) t}}{\lambda_1^0(|\xi|,a_1, a_2,b) -\lambda_2^0(|\xi|,a_1, a_2
    ,b)} \widehat{u_1}(\xi).
\end{align*}
The characteristic roots are given by
 \begin{align*}
   \lambda_1^0(|\xi|, a_1, a_2,b) &:= \frac{1}{2b} \left(-a_1|\xi|^{2\sigma_1} - a_2|\xi|^{2\sigma_2} + \sqrt{\left(a_1|\xi|^{2\sigma_1} + a_2|\xi|^{2\sigma_2}\right)^2-4b|\xi|^{2\sigma}}\right),\\
     \lambda_2^0(|\xi|, a_1, a_2, b) &:= -\frac{1}{2b} \left(a_1|\xi|^{2\sigma_1} + a_2|\xi|^{2\sigma_2} + \sqrt{\left(a_1|\xi|^{2\sigma_1} + a_2|\xi|^{2\sigma_2}\right)^2-4b|\xi|^{2\sigma}}\right).
 \end{align*}
 For simplicity, with loss of generality we can assume that $a_1=1$. Then, we denote $a_2=a$ again. We can rewrite $\lambda_1^0(|\xi|, 1, a, b) := \lambda_1^0(|\xi|,a,b)$ as
 \begin{align*}
    \lambda_1^0(|\xi|,a,b) &= \frac{-2|\xi|^{2(\sigma-\sigma_1)}}{\left(1+a|\xi|^{2(\sigma_2-\sigma_1)}\right) \left(1 + \displaystyle\sqrt{1-\frac{4b|\xi|^{2(\sigma-2\sigma_1)}}{\left(1+a|\xi|^{2(\sigma_2-\sigma_1)}\right)^2}}\right)}\\
     &=: -2|\xi|^{2(\sigma-\sigma_1)} \Gamma_1(|\xi|, a) (1+\Gamma_2(|\xi|,a, b))^{-1},
 \end{align*}
where we define the functions
\begin{align*}
    \Gamma_{1}(|\xi|, a) := \left(1+ a|\xi|^{2(\sigma_2-\sigma_1)}\right)^{-1} \quad \text{ and }\quad \Gamma_2(|\xi|, a,b) := \left(1-4b |\xi|^{2(\sigma-2\sigma_1)} (\Gamma_1(|\xi|,a))^2\right)^{1/2}.
\end{align*}
Due to the singularity of $\lambda_2^0(|\xi|,1, a,b) := \lambda_2^0(|\xi|,a,b)$ at $b=0$, to achieve the objectives of the paper we will consider
\begin{align*}
    \lambda_2^0(|\xi|,a,b) :=& -\frac{1}{2} \left(|\xi|^{2\sigma_1} + a|\xi|^{2\sigma_2} + \sqrt{\left(|\xi|^{2\sigma_1} + a|\xi|^{2\sigma_2}\right)^2-4b|\xi|^{2\sigma}}\right)\\
    =& -\frac{|\xi|^{2\sigma_1}}{2} (1+ a|\xi|^{2(\sigma_2-\sigma_1)}) (1+\Gamma_2(|\xi|,a,b)).
\end{align*}
Additionally, we also denote the following quantities:
\begin{itemize}
    \item $G(|\xi|, a, b) := \sqrt{\left(|\xi|^{2\sigma_1}+a|\xi|^{2\sigma_2}\right)^2-4b|\xi|^{2\sigma}} = |\xi|^{2\sigma_1} \left(1+a|\xi|^{2(\sigma_2-\sigma_1)}\right) \Gamma_2(|\xi|,a,b). $
    \vspace{0.5cm}
    \item $\widehat{K_0^1}(t, |\xi|, a, b) := (G(|\xi|, a,b))^{-1} \lambda_1^0(|\xi|,a,b) e^{\lambda_2^0(|\xi|,a,b)t} .$
    \vspace{0.5cm}
    \item  $\widehat{K_0^2}(t,|\xi|,a,b) := (G(|\xi|,a,b))^{-1} \lambda_2^0(|\xi|, a,b) e^{\lambda_1^0(|\xi|,a,b)t}.$
    \vspace{0.5cm}
    \item $\widehat{K_1^1}(t, |\xi|, a, b) := (G(|\xi|, a,b))^{-1} e^{\lambda_1^0(|\xi|,a,b)t}. $
    \vspace{0.5cm}
    \item $\widehat{K_1^2}(t,|\xi|,a,b) = (G(|\xi|,a,b))^{-1} e^{\lambda_2^0(|\xi|,a,b)t}.$ 
    \vspace{0.5cm}
    \item $\mathcal{A}_0^k(t,\xi) := 
    \begin{cases}
       \displaystyle \sum_{0 \leq j+m \leq k-1} \frac{1}{j!m!} \left(\frac{\partial^{j+m} \widehat{K_0^1}}{\partial a^j \partial b^m} (t,|\xi|,0,0) - \frac{\partial^{j+m}\widehat{K_0^2}}{\partial a^j \partial b^m} (t,|\xi|,0,0) \right)&\text{ if } k \geq 1,\\
       0 &\text{ if } k =0.
    \end{cases}$
    \vspace{0.5cm}
    \item
    $\mathcal{B}_0^k(t,\xi) = 
    \begin{cases}
    \displaystyle\sum_{0 \leq j+m \leq k-1} -\frac{1}{j!m!}\frac{\partial^{j+m}\widehat{K_0^2}}{\partial a^j \partial b^m}(t,|\xi|,0,0) &\text{ if } k \geq 1,\\
    0 &\text{ if } k = 0 .
    \end{cases}$ 
    \vspace{0.5cm}
    
    \item $\mathcal{A}_1^k(t,\xi) := 
    \begin{cases}
       \displaystyle \sum_{0 \leq j+m \leq k-1} \frac{1}{j!m!}\left(\frac{\partial^{j+m} \widehat{K_1^1}}{\partial a^j \partial b^m} (t,|\xi|,0,0) -   \frac{\partial^{j+m}\widehat{K_1^2}}{\partial a^j \partial b^m} (t,|\xi|,0,0) \right)&\text{ if } k \geq 1,\\
      0 &\text{ if } k =0.
    \end{cases}$
    \item 
$\mathcal{B}_1^k(t,\xi) =
\begin{cases}
    \displaystyle\sum_{0 \leq j+m \leq k-1} \frac{1}{j! m!} \frac{\partial^{j+m}\widehat{K_1^1}}{\partial a^j \partial b^m}(t,|\xi|,0,0) &\text{ if } k \geq 1,\\
    0 &\text{ if } k = 0.
    \end{cases}$
\end{itemize}

\vspace{0.5cm}
  As usual, $H^{s}$ and $\dot{H}^{s}$, with $s \ge 0$, denote Bessel and Riesz potential spaces based on $L^2$ spaces. Here $\big<D\big>^{s}$ and $|D|^{s}$ stand for the differential operators with symbols $\big<\xi\big>^{s}$ and $|\xi|^{s}$, respectively. Moreover, for any $\nu \in \R$, we also denote by $[\nu]^+:= \max\{\nu,0\}$, its positive part.

Finally, we fix the constant
$$ \delta := \min\{\sigma_2-\sigma_1, \sigma-2\sigma_1 \} =
\begin{cases}
    \sigma_2-\sigma_1
 &\text{ if } \sigma_2 +\sigma_1 < \sigma \\
 \sigma-2\sigma_1 &\text{ if } \sigma_2 +\sigma_1 \geq \sigma
 \end{cases} $$
and denote the quantity
$$ P_1 := \displaystyle\int_{\mathbb{R}^n} u_1(x) dx. $$

\subsection{Main results}
Let us now state the main results which will be proved in this paper.
\begin{theorem}[\textbf{Asymptotic profile with $\sigma_1 > 0$}]\label{Linear_Asym}
    Let $n > 4\sigma_1, \,s \geq 0$ and $k \in \mathbb{N}$. Assuming that the initial data $(u_0, u_1)$ satisfy
    \begin{align*}
        (u_0, u_1) \in \mathcal{D}_s := (H^{s} \cap L^1) \times (H^{[s-2\sigma_2]^+} \cap L^1). 
    \end{align*}
    Then, the solution to \eqref{Main.Eq.1} satisfy 
    \begin{align}
        &\bigg\| |\xi|^s\widehat{u}(t,\xi) - |\xi|^s \mathcal{A}_0^k(t,\xi) \widehat{u_0}(\xi) - |\xi|^s\mathcal{A}_1^k(t,\xi) \widehat{u_1}(\xi)\bigg\|_{L^2}\notag\\
        &\hspace{5cm}\lesssim (1+t)^{-\frac{n}{4(\sigma-\sigma_1)}-\frac{s}{2(\sigma-\sigma_1)}+\frac{\sigma_1}{\sigma-\sigma_1}-k\frac{\delta}{\sigma-\sigma_1}} \|(u_0, u_1)\|_{\mathcal{D}_s}, \label{UpperEs2.1}
    \end{align}
    for all $t \geq 1$. Moreover,  if we assume an additional condition $P_1 \neq 0$, then there exists a positive constant 
$C_1:= C_1 \left( P_1, \|(u_0, u_1)\|_{\mathcal{D}_s}\right) $ such that the following estimate hold:
        \begin{align}
        &\bigg\|  |\xi|^s\widehat{u}(t,\xi) - |\xi|^s\mathcal{A}_0^k(t,\xi)\widehat{u_0}(\xi) - |\xi|^s \mathcal{A}_1^k(t,\xi) \widehat{u_1}(\xi)\bigg\|_{L^2} \nonumber \\
        &\hspace{5cm} \geq C_1 (1+t)^{-\frac{n}{4(\sigma-\sigma_1)}-\frac{s}{2(\sigma-\sigma_1)}+\frac{\sigma_1}{\sigma-\sigma_1}-k\frac{\delta}{\sigma-\sigma_1}}, \label{LowerEs2.1}
    \end{align}
    for all $t \geq 1$.
\end{theorem}
\begin{remark}
\fontshape{n}
\selectfont
    By choosing $s=0$, from \eqref{UpperEs2.1} let us describe the asymptotic profile of solutions to (\ref{Main.Eq.1}) in the case $\sigma_1 > 0$ as follows:
\begin{equation} \label{Optimal_Rep}
    \widehat{u}(t,\xi) \sim  \mathcal{A}_0^k(t,\xi) u_0(\xi) +  \mathcal{A}_1^k(t,\xi)u_1(\xi). 
\end{equation}
To be specific, for $k=1$ we calculate
\begin{align*}
    \mathcal{A}_0^1(t,\xi) = -|\xi|^{2(\sigma-2\sigma_1)} e^{-|\xi|^{2\sigma_1}t} + e^{-|\xi|^{2(\sigma-\sigma_1)}} \text{ and }\,\, \mathcal{A}_1^1(t,\xi) = \frac{e^{-|\xi|^{2(\sigma-\sigma_1)}t}}{|\xi|^{2\sigma_1}} - \frac{e^{-|\xi|^{2\sigma_1}t}}{|\xi|^{2\sigma_1}}.
\end{align*}
For $k=2$, we obtain
\begin{align*}
    \mathcal{A}_0^2(t,\xi) &= -|\xi|^{2(\sigma-2\sigma_1)} e^{-|\xi|^{2\sigma_1}t} \left(1-2|\xi|^{2(\sigma_2-\sigma_1)} +3 |\xi|^{2(\sigma-2\sigma_1)}\right) \\
    &\hspace{2cm}+  t|\xi|^{2(\sigma-\sigma_1)} e^{-|\xi|^{2\sigma_1} t} \left(|\xi|^{2(\sigma_2-\sigma_1)}-|\xi|^{2(\sigma-2\sigma_1)}\right) \\
    &\quad+ e^{-|\xi|^{2(\sigma-\sigma_1)}t} \left(1+|\xi|^{\sigma-2\sigma_1}\right)\\
    &\hspace{2cm} + t|\xi|^{2(\sigma-\sigma_1)} e^{-|\xi|^{2(\sigma-\sigma_1)}t}\left(|\xi|^{2(\sigma_2-\sigma_1)}-|\xi|^{2(\sigma-2\sigma_1)}\right)
    \end{align*}
    and 
    \begin{align*}
    \mathcal{A}_1^2(t,\xi) &= \frac{e^{-|\xi|^{2(\sigma-\sigma_1)}t}}{|\xi|^{2\sigma_1}} \left(1-|\xi|^{2(\sigma_2-\sigma_1)}+2|\xi|^{2(\sigma-2\sigma_1)} \right) \\
    &\hspace{2cm}+ t|\xi|^{2(\sigma-\sigma_1)} \frac{e^{-|\xi|^{2(\sigma-\sigma_1)}t}}{|\xi|^{2\sigma_1}} \left(|\xi|^{2(\sigma_2-\sigma_1)}-|\xi|^{2(\sigma-2\sigma_1)}\right)\\
    &\qquad- \frac{e^{-|\xi|^{2\sigma_1}t}}{|\xi|^{2\sigma_1}} \left(1-|\xi|^{2(\sigma_2-\sigma_1)}+2|\xi|^{2(\sigma-2\sigma_1)}\right)\\
    &\hspace{2cm}\quad -te^{-|\xi|^{2\sigma_1}t} \left(-|\xi|^{2(\sigma_2-\sigma_1)}+|\xi|^{2(\sigma-2\sigma_1)}\right). 
\end{align*}
From this observation, we can see the influence of the parameter $\sigma_2$ representing ``$\sigma$-evolution like models'' on the asymptotic behavior of solutions with respect to the second-order. When comparing the above results with (\ref{1_order_Asym}), we may observe that some redundant quantities appear. The advantage of the solution representation \eqref{Optimal_Rep} is that it allows us to conclude the precise decay time of the error term for large time, namely,
\begin{align*}
    \left\| \widehat{u}(t,\xi) - \mathcal{A}_0^k(t,\xi) \widehat{u_0}(\xi) - \mathcal{A}_1^k(t,\xi) \widehat{u_1}(\xi)\right\|_{L^2} \sim (1+t)^{-\frac{n}{4(\sigma-\sigma_1)}+\frac{\sigma_1}{\sigma-\sigma_1}-k\frac{\delta}{\sigma-\sigma_1}},
\end{align*}
which is really new in a comparison with all previous studies. 
\end{remark}

\begin{theorem}[\textbf{Asymptotic profile with $\sigma_1 = 0$}]\label{Linear_Asym_1}
    Let $n \geq 1, s \geq 0$ and $k \in \mathbb{N}$. Assuming that the initial data $(u_0, u_1)$ satisfy
    \begin{align*}
        (u_0, u_1) \in \mathcal{D}_s := (H^{s} \cap L^1) \times (H^{[s-2\sigma_2]^+} \cap L^1). 
    \end{align*}
    Then, the solution to the problem (\ref{Main.Eq.1}) with $\sigma_1 = 0$ satisfy 
    \begin{align}
        &\bigg\||\xi|^s  \widehat{u}(t,\xi) - |\xi|^s \mathcal{B}_0^k(t,\xi) \widehat{u_0}(\xi) - |\xi|^s\mathcal{B}_1^k(t,\xi) \widehat{u_1}(\xi)\bigg\|_{L^2} \lesssim (1+t)^{-\frac{n}{4\sigma}-\frac{s}{2\sigma}-k\frac{\sigma_2}{\sigma}} \|(u_0, u_1)\|_{\mathcal{D}_s}, \label{UpperEs2.2}
    \end{align}
    for all $t \geq 1$. Moreover, if we assume an additional condition $P_1 \neq 0$, there exists a positive constant 
$C_2 := C_2\left(P_1, \|(u_0, u_1)\|_{\mathcal{D}_s}\right) $ such that the following estimate hold:
        \begin{align}
        &\bigg\|  |\xi|^s\widehat{u}(t,\xi) - |\xi|^s\mathcal{B}_0^k(t,\xi)\widehat{u_0}(\xi) - |\xi|^s\mathcal{B}_1^k(t,\xi) \widehat{u_1}(\xi)\bigg\|_{L^2} \geq C_2 (1+t)^{-\frac{n}{4\sigma}-\frac{s}{2\sigma}-k\frac{\sigma_2}{\sigma}}, \label{LowerEs2.2}
    \end{align}
    for all $t \geq 1$.
\end{theorem}
\begin{remark}
\fontshape{n}
\selectfont
    By choosing $s = 0$, we can describe the asymptotic profile of solutions to (\ref{Main.Eq.1}) in the case $\sigma_1 = 0$ as follows:
    \begin{align*}
     \widehat{u}(t,\xi) \sim  \mathcal{B}_0^k(t,\xi) u_0(\xi) + 
     \mathcal{B}_1^k(t,\xi) u_1(\xi).
\end{align*}
Moreover, we also obtain the sharp estimate for $k$-th order asymptotic expansion error of solutions for large time as follows:
\begin{align*}
    &\bigg\|  |\xi|^s\widehat{u}(t,\xi) - |\xi|^s\mathcal{B}_0^k(t,\xi)\widehat{u_0}(\xi) - |\xi|^s\mathcal{B}_1^k(t,\xi) \widehat{u_1}(\xi)\bigg\|_{L^2} \sim (1+t)^{-\frac{n}{4\sigma}-\frac{s}{2\sigma}-k\frac{\sigma_2}{\sigma}}.
\end{align*}
Specifically, for $k=1$, we calculate
\begin{align*}
    \mathcal{B}_0^1(t,\xi) = \mathcal{B}_1^1(t,\xi) = e^{-|\xi|^{2\sigma}t}
\end{align*}
and for $k=2$, we obtain
\begin{align*}
    \mathcal{B}_0^2(t,\xi) &= e^{-|\xi|^{2\sigma}t} \left(1 +|\xi|^{2\sigma}\right) + t|\xi|^{2\sigma} e^{-|\xi|^{2\sigma}t}\left(|\xi|^{2\sigma_2}-|\xi|^{2\sigma}\right),\\
    \mathcal{B}_1^2(t,\xi) &= e^{-|\xi|^{2\sigma}t} \left(1-|\xi|^{2\sigma_2}+2|\xi|^{2\sigma} \right)+ t|\xi|^{2\sigma} e^{-|\xi|^{2\sigma}t}\left(|\xi|^{2\sigma_2}-|\xi|^{2\sigma}\right).
\end{align*}
\end{remark}
\begin{remark}
\fontshape{n}
\selectfont
    The difference between  Theorem \ref{Linear_Asym} and Theorem \ref{Linear_Asym_1} is that the quantities 
    \begin{align*}
        \displaystyle\frac{\partial^{j+m}}{\partial a^j \partial b^m} \widehat{K_0^1}(t,|\xi|, 0,0) \text{ and } \frac{\partial^{j+m}}{\partial a^j \partial b^m} \widehat{K_1^2}(t,|\xi|, 0,0)
    \end{align*}
  appearing in Theorem \ref{Linear_Asym} in the special case $\sigma_1 = 0$ (corresponding to frictional damping) have been eliminated for all $j, m \in \mathbb{N}$. This will be explained in detail in the next section.
\end{remark}

\section{Proofs of main results} \label{section2}
\subsection{Preliminaries}
At the beginning of this section, let us apply the Fourier transform to \eqref{Main.Eq.1} to derive
\begin{align}
    \begin{cases}
        \widehat{u}_{tt} + |\xi|^{2\sigma} \widehat{u}+ \big(|\xi|^{2\sigma_1} + |\xi|^{2\sigma_2}\big) \widehat{u}_t = 0 &\xi \in \mathbb{R}^n, \, t > 0, \\
        \widehat{u}(0,\xi) = \widehat{u}_0(\xi), \quad \widehat{u}_t(0,\xi) = \widehat{u}_1(\xi) &\xi \in \mathbb{R}^n.
    \end{cases} \label{Main.Eq.1_Fourier}
\end{align}
The representation formula of solutions to (\ref{Main.Eq.1_Fourier}) reads as follows:
$$ \widehat{u}(t,\xi)= \widehat{K_0}(t,\xi) \widehat{u_0}(\xi) + \widehat{K_1}(t,\xi) \widehat{u_1}(\xi), $$
that is, the solutions to (\ref{Main.Eq.1}) can be written by
\begin{equation}
    u(t,x) = K_0(t,x) \ast_x u_0(x) + K_1(t,x) \ast_x u_1(x), \label{represen1}
\end{equation}
where
\begin{align}
   \widehat{K_0}(t,\xi) &:= \frac{\lambda_1(|\xi|) e^{\lambda_2(|\xi|) t}-\lambda_2(|\xi|) e^{\lambda_1(|\xi|)t}}{\lambda_1(|\xi|)-\lambda_2(|\xi|)} = \widehat{K_0^1}(t,|\xi|,1,1) - \widehat{K_0^2}(t,|\xi|,1,1), \label{repre1}\\  \widehat{K_1}(t,\xi) &:= \frac{e^{\lambda_1(|\xi|) t}-e^{\lambda_2(|\xi|) t}}{\lambda_1(|\xi|) -\lambda_2(|\xi|)} = \widehat{K_1^1}(t,|\xi|,1,1) - \widehat{K_1^2}(t,|\xi|,1,1), \label{repre2}
\end{align}
and the characteristic roots $\lambda_{1}(|\xi|) := \lambda_{1}^0(|\xi|,1,1)$ and $\lambda_2(|\xi|) := \lambda_2^0(|\xi|,1,1)$ are given by
\begin{equation*}
\lambda_{1,2}(|\xi|) =
\begin{cases}
    \f{1}{2}\left(-|\xi|^{2\sigma_1} - |\xi|^{2\sigma_2} \pm \sqrt{(|\xi|^{2\sigma_1} + |\xi|^{2\sigma_2})^{2} - 4 |\xi|^{2\sigma}}\right) & \text { if } |\xi| \in \mathbb{R}_+ \setminus \Omega, \\
    \\
     \f{1}{2}\left(-|\xi|^{2\sigma_1} - |\xi|^{2\sigma_2} \pm i \sqrt{4 |\xi|^{2\sigma}-(|\xi|^{2\sigma_1} + |\xi|^{2\sigma_2})^{2} }\right) & \text { if } |\xi| \in \Omega,  
\end{cases}
\end{equation*}
where $\Omega = \left\{r \in \mathbb{R}_+:  r^{2\sigma_1} +  r^{2\sigma_2} < 2 r^{\sigma} \right\}$. For simplicity, in the sequels let us denote
\begin{align*}
    & \widehat{K_0^1}(t,|\xi|)= \widehat{K_0^1}(t,|\xi|,1,1), \quad \widehat{K_0^2}(t,|\xi|)= \widehat{K_0^2}(t,|\xi|,1,1), \\
    & \widehat{K_1^1}(t,|\xi|)= \widehat{K_1^1}(t,|\xi|,1,1), \quad \widehat{K_1^2}(t,|\xi|)= \widehat{K_1^2}(t,|\xi|,1,1).
\end{align*}
Because of $0 \leq \sigma_1 <\sigma/2 < \sigma_2 \leq \sigma$, there exists a sufficiently small constant $\varepsilon^*>0$ such that
\begin{equation*}
    (0, \varepsilon^*)\cup \left(\frac{1}{\varepsilon^*},\ity\right) \subset \Omega.
\end{equation*}
Then, taking account of the cases of small and large frequencies separately we conclude that
\begin{align}
     &\lambda_{1} \sim -|\xi|^{2(\sigma-\sigma_1)},\quad 
    \lambda_{2} \sim -|\xi|^{2\sigma_1}, \quad \lambda_1-\lambda_2 \sim |\xi|^{2\sigma_1}\quad \text{ for } |\xi| \leq \varepsilon^*, \notag\\
    &\lambda_{1} \sim -|\xi|^{2(\sigma-\sigma_2)} ,\quad \lambda_{2} \sim -|\xi|^{2\sigma_2}, \quad \lambda_1-\lambda_2 \sim |\xi|^{2\sigma_2} \quad \text{ for }|\xi| \geq \frac{1}{\varepsilon^*}.\label{Rela1}
\end{align} 
Let $\chi_{p}= \chi_{p}(r) \in \mathcal{C}^{\infty}([0, \infty))$ with $p\in\{\rm L,H\}$ be smooth cut-off functions having the following properties:
\begin{align*}
&\chi_{\rm L}(r)=
\begin{cases}
1 &\quad \text{ if }r\le \varepsilon^*/2, \\
0 &\quad \text{ if }r\ge \varepsilon^*
\end{cases} 
\text{ and } \chi_{\rm H}(r)= 1- \chi_{\rm L}(r).
\end{align*}
It is obvious to see that $\chi_{\rm H}(r)= 1$ if $r \geq \varepsilon^*$ and $\chi_{\rm H}(r)= 0$ if $r \le \varepsilon^*/2$.
In order to prove Theorems \ref{Linear_Asym} and \ref{Linear_Asym_1}, the following auxiliary results come into play.
\begin{lemma}[Faà di Bruno’s formula]\label{FadiBruno'sformula}
Let $\omega_1(x) = f_1\big(g_1(x)\big)$ with $x \in \mathbb{R}$. Then, we have
\begin{align*}
 \frac{\partial^m}{\partial x^m}\omega_1(x)= \sum_{\beta_1+ 2 \beta_2+...+ m \beta_m =m} \frac{m!}{\beta_1! \beta_2! \cdots \beta_m!} f_1^{(\beta_1+\beta_2+\cdots+\beta_m)}\big(g_1(x)\big) \prod_{j=1}^m \bigg(\frac{1}{j!}\frac{\partial^j}{\partial x^j} g_1(x) \bigg)^{\beta_j}.
 \end{align*}
Furthermore, let $\omega(x_1, x_2)= f( g(x_1, x_2))$ with $(x_1, x_2)\in \mathbb{R}^2$. Then, the following formula holds for all $j, m \in \mathbb{N}$:
\begin{align*}
\frac{\partial^{j+m}}{\partial x_1^j \partial x_2^m} \omega(x_1, x_2) = \sum_{\ell=1}^{j+m} f^{(\ell)}(g(x_1, x_2)) \sum_{h=1}^{j+m} \sum_{\mathcal{S}_{\ell, h}^{j,m}( a,b)} j! m! \prod_{\rho = 1}^h \frac{1}{a_{\rho} !} \left(\frac{1}{b_{1, \rho}! b_{2, \rho}! }\frac{\partial^{b_{1,\rho}+b_{2, \rho}} }{\partial x_1^{b_{1, \rho}} \partial x_2^{b_{2, \rho}}}g(x_1, x_2)\right)^{a_{\rho}}, 
\end{align*}
where 
\begin{align*}
    \mathcal{S}_{ \ell, h}^{j,m} (a,b) := \bigg\{(a_{\rho}, b_{1,\rho}, b_{2,\rho}),  1 \leq \rho \leq h \text{ satisfying } a_{\rho} > 0,  \sum_{\rho = 1}^h a_{\rho} = \ell, \sum_{\rho = 1}^h a_{\rho} b_{1, \rho} = j, \sum_{\rho = 1}^h a_{\rho} b_{2, \rho} = m \\ \quad
     \text{ and } (b_{1, \rho_1}, b_{2, \rho_1}) \prec (b_{1, \rho_2}, b_{2, \rho_2}) \text{ for all } 1 \leq \rho_1 < \rho_2 \leq h\bigg\}.
\end{align*}
Here we write $ (b_{1, \rho_1}, b_{2, \rho_1}) \prec (b_{1, \rho_2}, b_{2, \rho_2})$ provided that one of the following holds:

\begin{itemize}
    \item [i)]  $b_{1, \rho_1} < b_{1, \rho_2}$,
    \item [ii)] $b_{1, \rho_1} = b_{1, \rho_2}$ and $b_{2, \rho_1} < b_{2, \rho_2}$.
\end{itemize}

\end{lemma}
The first formula of Lemma \ref{FadiBruno'sformula} appears quite frequently in the recent papers (e.g., \cite{DaoReissig2019_1, Michihisa2021}), while the second one can be deduced from Theorem 2.1 in \cite{Constantine1996}.

%..........................................................................
\subsection{Estimates for low frequencies}
By Lemma \ref{FadiBruno'sformula}, we are going to prove the following lemmas.
\begin{lemma}\label{Lemma2.2}
Let $j ,m \in \mathbb{N}$ and $0 \leq \sigma_1 < \sigma/2 < \sigma_2 \leq \sigma$. Then, we have the following estimates:
    \begin{align}
        \left|\frac{\partial^{j}}{\partial a^j} (\Gamma_1(|\xi|,a))^\alpha\right| &\lesssim |\xi|^{2j (\sigma_2-\sigma_1)},\label{estimate2.2.1} \\
        \left|\frac{\partial^{j+m}}{\partial a^j \partial b^m} \Gamma_2(|\xi|,a,b)\right| &\lesssim |\xi|^{2j(\sigma_2-\sigma_1)+2m(\sigma-2\sigma_1)}, \label{estimate2.2.2}
    \end{align}
    for all $ (|\xi|, a,b) \in (0, \varepsilon^*]\times [0,1] \times [0,1]$ and $\alpha > 0$.
      Moreover, there exist some constants $C_{\alpha, j}, C_{j,m}$ satisfying
    \begin{align}
        \frac{\partial^j}{\partial a^j} (\Gamma_1(|\xi|,a))^{\alpha} \bigg|_{a = 0} &= C_{\alpha, j} |\xi|^{2j(\sigma_2-\sigma_1)}, \label{Re2.2.1}\\
        \frac{\partial^{j+m}}{\partial a^j \partial b^m} \Gamma_2(|\xi|, a, b)\bigg|_{(a,b) = (0,0)} &= C_{j,m} |\xi|^{2j(\sigma_2-\sigma_1)+2m(\sigma-2\sigma_1)}. \label{Re2.2.2}
    \end{align}
\end{lemma}
\begin{proof}
     At first, with $j \geq 1$, we can see that
    \begin{align*}
        \frac{\partial^j}{\partial a^j} (\Gamma_1(|\xi|,a))^{\alpha} &= \frac{\partial^j}{\partial a^j} \left(1+a|\xi|^{2(\sigma_2-\sigma_1)}\right)^{-\alpha}\\
        &=(-\alpha)(-\alpha-1)\cdots(-\alpha-j+1) |\xi|^{2j(\sigma_2-\sigma_1)}\left(1+a|\xi|^{2(\sigma_2-\sigma_1)}\right)^{-\alpha-j}, 
    \end{align*}
    which gives the conclusions (\ref{estimate2.2.1}) and (\ref{Re2.2.1}) for all $j \geq 0$. Next, we will prove (\ref{estimate2.2.2}) and (\ref{Re2.2.2}). Due to the fact that
    \begin{align*}
        &\frac{\partial^m}{\partial b^m} \Gamma_2(|\xi|,a,b) \\
        &\qquad = \prod_{h=0}^{m-1}(1/2-h) (-4)^m |\xi|^{2m(\sigma-2\sigma_1)} (\Gamma_1 (|\xi|,a))^{2m} \left(1-4b|\xi|^{2(\sigma-2\sigma_1)} (\Gamma_1(|\xi|,a))^2\right)^{1/2-m}\\
        &\qquad = \prod_{h=0}^{m-1}(1/2-h)(-4)^m |\xi|^{2m(\sigma-2\sigma_1)} (\Gamma_1 (|\xi|,a))^{2m} (\Gamma_2 (|\xi|,a, b))^{1-2m}
    \end{align*}
    for all $m \geq 1$, this leads to
    \begin{align}
        &\frac{\partial^{j+m}}{\partial a^j \partial b^m} \Gamma_2(|\xi|, a,b) \nonumber \\
        &\qquad = C_m |\xi|^{2m(\sigma-2\sigma_1)} \sum_{h=0}^j \frac{j!}{h! (j-h)!} \frac{\partial^{j-h}}{\partial a^{j-h}} (\Gamma_1(|\xi|,a))^{2m} \frac{\partial^h}{\partial a^h} (\Gamma_2(|\xi|,a, b))^{1-2m}. \label{Re2.2.3} 
    \end{align}
    Applying the first formula of Lemma \ref{FadiBruno'sformula} we arrive at
    \begin{align*}
        &\frac{\partial^h}{\partial a^h} (\Gamma_2(|\xi|,a, b))^{1-2m} \\
        &\qquad = \sum_{\gamma_1 + 2\gamma_2+\cdots+ h \gamma_h = h} \frac{h!}{\gamma_1 ! \gamma_2 ! \cdots \gamma_h !} \frac{\partial^{\sum_{\rho =1}^h \gamma_{\rho}}}{\partial q^{\sum_{\rho=1}^h \gamma_{\rho}}} (1-4q)^{1/2-m} \bigg|_{q = b|\xi|^{2(\sigma-2\sigma_1)}(\Gamma_1(|\xi|,a))^2} \\
        &\qquad\qquad \times \prod_{\rho= 1}^h \left(\frac{1}{\rho !}  b|\xi|^{2(\sigma-2\sigma_1)}\frac{\partial^{\rho}}{\partial a^{\rho}} (\Gamma_1(|\xi|,a))^2 \right)^{\gamma_{\rho}}.
    \end{align*}
    Therefore, we obtain the following relation:
    \begin{align*}
        \frac{\partial^h}{\partial a^h} (\Gamma_2(|\xi|,a, b))^{1-2m} \bigg|_{(a,b) = (0,0)} = 
        \begin{cases}
        1 &\text{ if } h = 0 \\
            0 &\text{ if } h \geq 1.
        \end{cases}
    \end{align*}
    Combining this with the estimate (\ref{estimate2.2.1}) for $\alpha = 2$ leads to
\begin{align*}
    \left|\frac{\partial^h}{\partial a^h} (\Gamma_2 (|\xi|,a,b))^{1-2m}\right| \lesssim |\xi|^{2h(\sigma_2-\sigma_1)}
\end{align*}
 for all  $(|\xi|,a,b) \in (0, \varepsilon^*] \times [0,1] \times [0,1]$  and $ h \geq 0$. By applying again the estimate (\ref{estimate2.2.1}) for $\alpha = 2m$ and the above relations to (\ref{Re2.2.3}), we derive (\ref{estimate2.2.2}) and (\ref{Re2.2.2}). Thus, Lemma \ref{Lemma2.2} has been established.
\end{proof}
\begin{lemma}\label{lemma2.3}
    Let $j ,m \in \mathbb{N}$ and $0 \leq \sigma_1 < \sigma/2 < \sigma_2 \leq \sigma$. Then, we have the following estimates:
    \begin{align}
        \left|\frac{\partial^{j+m}}{\partial a^j \partial b^m} \lambda_1^0(|\xi|,a,b)\right| &\lesssim |\xi|^{2(\sigma-\sigma_1)+2j (\sigma_2-\sigma_1)+2m(\sigma-2\sigma_1)},\label{estimate2.3.1} \\
        \left|\frac{\partial^{j+m}}{\partial a^j \partial b^m} e^{\lambda_1^0(|\xi|,a,b) t}\right| &\lesssim e^{-c|\xi|^{2(\sigma-\sigma_1)}t}|\xi|^{2j(\sigma_2-\sigma_1)+2m(\sigma-2\sigma_1)}, \label{estimate2.3.2}\\
        \left|\frac{\partial^{j+m}}{\partial a^j \partial b^m} \lambda_2^0(|\xi|,a,b)\right| &\lesssim |\xi|^{2\sigma_1 + 2j(\sigma_2-\sigma_1)+2m(\sigma-2\sigma_1)} \label{estimate2.3.3},\\
        \left|\frac{\partial^{j+m}}{\partial a^j \partial b^m} e^{\lambda_2^0(|\xi|, a,b)t}\right| &\lesssim e^{-c|\xi|^{2\sigma_1}t} |\xi|^{2j(\sigma_2-\sigma_1)+2m(\sigma-2\sigma_1)}, \label{estimate2.3.4}
    \end{align}
    for all $ (|\xi|, a,b) \in (0, \varepsilon^*]\times [0,1] \times [0,1]$ and $c$ is a suitable positive constant.
      Moreover, there exist some constants $C_{1, j,m}, C_{2,j,m}, C_{1,h,j,m}^*, C_{2, h,j,m}^*$ satisfying
    \begin{align}
        \frac{\partial^{j+m}}{\partial a^j \partial b^m} \lambda_1^0(|\xi|,a,b)\bigg|_{(a,b) = (0,0)} &= C_{1,j,m}|\xi|^{2(\sigma-\sigma_1)+2j (\sigma_2-\sigma_1)+2m(\sigma-2\sigma_1)}, \label{Re2.3.1}\\
        \frac{\partial^{j+m}}{\partial a^j \partial b^m} e^{\lambda_1^0(|\xi|,a,b) t} \bigg|_{(a,b) = (0,0)} &=  e^{-|\xi|^{2(\sigma-\sigma_1)}t} |\xi|^{2j(\sigma_2-\sigma_1)+2m(\sigma-2\sigma_1)} \sum_{h=1}^{j+m} C_{1,h,j,m}^* (|\xi|^{2(\sigma-\sigma_1)}t )^h, \label{Re2.3.2} \\
        \frac{\partial^{j+m}}{\partial a^j \partial b^m} \lambda_2^0(|\xi|,a,b)\bigg|_{(a,b) = (0,0)} &= C_{2,j,m}|\xi|^{2\sigma_1+2j (\sigma_2-\sigma_1)+2m(\sigma-2\sigma_1)}, \label{Re2.3.3}\\
        \frac{\partial^{j+m}}{\partial a^j \partial b^m} e^{\lambda_2^0(|\xi|,a,b) t} \bigg|_{(a,b) = (0,0)} &=  e^{-|\xi|^{2\sigma_1}t} |\xi|^{2j(\sigma_2-\sigma_1)+2m(\sigma-2\sigma_1)} \sum_{h=1}^{j+m} C_{2,h,j,m}^* (|\xi|^{2\sigma_1}t )^h. \label{Re2.3.4}
    \end{align}
\end{lemma}
\begin{proof}
    First, we recall the representation of the term $\lambda_1^0(|\xi|,a,b)$ as follows:
    \begin{align*}
        \lambda_1^0(|\xi|,a,b) = -2|\xi|^{2(\sigma-\sigma_1)} \Gamma_1(|\xi|,a) (1+ \Gamma_2(|\xi|,a,b))^{-1}.
    \end{align*}
    Using Leibniz's formula one arrives at
    \begin{align}
        &\frac{\partial^{j+m}}{\partial a^j \partial b^m} \lambda_1^0(|\xi|,a,b) \nonumber \\
        &\qquad = -2|\xi|^{2(\sigma-\sigma_1)} \sum_{h=0}^j \frac{j!}{h!(j-h)!} \frac{\partial^{j-h}}{\partial a^{j-h}} \Gamma_1(|\xi|,a) \frac{\partial^{h+m}}{\partial a^h \partial b^m} (1+\Gamma_2(|\xi|, a,b))^{-1}. \label{Re2.3.5}
    \end{align}
    We employ the second formula of Lemma \ref{FadiBruno'sformula} with $h+m \geq 1$ to obtain
\begin{align*}
&\frac{\partial^{h+m}}{\partial a^h \partial b^m} (1+\Gamma_2(|\xi|,a,b))^{-1} \\
&\qquad= \sum_{\ell=1}^{h+m} \frac{\partial^{\ell}}{\partial q^{\ell}} (1+q)^{-1}\bigg|_{q= \Gamma_2(|\xi|,a,b)}\\
&\hspace{3cm} \times \sum_{\theta=1}^{h+m} \sum_{\mathcal{S}_{\ell, \theta}^{h,m}( \beta,\gamma)} h! m! \prod_{\rho = 1}^\theta \frac{1}{\beta_{\rho} !} \left(\frac{1}{\gamma_{1, \rho}! \gamma_{2, \rho}! }\frac{\partial^{\gamma_{1,\rho}+\gamma_{2, \rho}} }{\partial a^{\gamma_{1, \rho}} \partial b^{\gamma_{2, \rho}}}\Gamma_2(|\xi|, a,b)\right)^{\beta_{\rho}}, 
\end{align*}
where $\mathcal{S}_{\ell, \theta}^{h,m}(\beta, \gamma)$ is defined as in Lemma \ref{FadiBruno'sformula}.
Using the estimate (\ref{estimate2.2.2}) and the relation (\ref{Re2.2.2}) for the term
\begin{align*}
    \displaystyle\frac{\partial^{\gamma_{1,\rho}+\gamma_{2, \rho}} }{\partial a^{\gamma_{1, \rho}} \partial b^{\gamma_{2, \rho}}}\Gamma_2(|\xi|, a,b),
\end{align*}
we achieve
\begin{align*}
    \left|\frac{\partial^{h+m}}{\partial a^h \partial b^m} (1+ \Gamma_2(|\xi|,a,b))^{-1}\right| &\lesssim |\xi|^{2h(\sigma_2-\sigma_1)+2m(\sigma-2\sigma_1)},\\
    \frac{\partial^{h+m}}{\partial a^h \partial b^m} (1+ \Gamma_2(|\xi|,a,b))^{-1}\bigg|_{(a,b)=(0,0)} &= C_{h,m}^{'} |\xi|^{2h(\sigma_2-\sigma_1)+2m(\sigma-2\sigma_1)}.
\end{align*}
Here, we note that these relations still hold when $(h, m) = (0, 0)$. By combining them with (\ref{estimate2.2.1}) and (\ref{Re2.2.1}) for $\alpha = 1$, we derive (\ref{estimate2.3.1}) and (\ref{Re2.3.1}) from the relation (\ref{Re2.3.5}). Next, to indicate (\ref{estimate2.3.2}) and (\ref{Re2.3.2}) we apply again the second formula of Lemma \ref{FadiBruno'sformula} to get
\begin{align*}
    &\frac{\partial^{j+m}}{\partial a^j \partial b^m} e^{\lambda_1^0(|\xi|,a,b)t} \\
    &\quad= \sum_{\ell = 1}^{j+m} \frac{\partial^{\ell}}{\partial q ^{\ell}} e^q\bigg|_{q= \lambda_1^0(|\xi|,a,b)t}\sum_{\theta=1}^{j+m} \sum_{\mathcal{S}_{\ell, \theta}^{j,m}( \beta,\gamma)} j! m! \prod_{\rho = 1}^\theta \frac{1}{\beta_{\rho} !} \left(\frac{1}{\gamma_{1, \rho}! \gamma_{2, \rho}! }\frac{\partial^{\gamma_{1,\rho}+\gamma_{2, \rho}} }{\partial a^{\gamma_{1, \rho}} \partial b^{\gamma_{2, \rho}}}\lambda_1^0(|\xi|, a,b) t\right)^{\beta_{\rho}}.
\end{align*}
By combining this with (\ref{Re2.3.1}), one immediately finds
\begin{align*}
    &\frac{\partial^{j+m}}{\partial a^j \partial b^m} e^{\lambda_1^0(|\xi|,a,b)t} \bigg|_{(a,b) = (0,0)} \\
    \vspace{0.3cm}
    &\qquad= C_{j,m} e^{-|\xi|^{2(\sigma-\sigma_1)}t} |\xi|^{2j(\sigma_2-\sigma_1)+2m(\sigma-2\sigma_1)}\\
    &\hspace{2cm}\times \sum_{\ell=1}^{j+m} \sum_{\theta = 1}^{j+m} \sum_{\mathcal{S}_{\ell, \theta}^{j,m}( \beta, \gamma)} j! m! \left(|\xi|^{2(\sigma-\sigma_1)}t\right)^{\sum_{\rho=1}^\theta \beta_{\rho}}\prod_{\rho = 1}^\theta \frac{1}{\beta_{\rho} !} \left(\frac{1}{\gamma_{1, \rho}! \gamma_{2, \rho}! } \right)^{\beta_{\rho}}\\
    &\qquad= e^{-|\xi|^{2(\sigma-\sigma_1)}t} |\xi|^{2j(\sigma_2-\sigma_1)+2m(\sigma-2\sigma_1)} \sum_{\ell=1}^{j+m} C_{l,j,m} \left(|\xi|^{2(\sigma-\sigma_1)}t\right)^{\ell}.
\end{align*}
Moreover, from the estimate (\ref{estimate2.3.1}) we derive
\vspace{0.3cm}
\begin{align*}
    &\left|\frac{\partial^{j+m}}{\partial a^j \partial b^m} e^{\lambda_1^0(|\xi|,a,b)t}\right|\\
    &\qquad\lesssim e^{-c_1 |\xi|^{2(\sigma-\sigma_1)}t} |\xi|^{2j(\sigma_2-\sigma_1)+2m(\sigma-2\sigma_1)}\\
    &\hspace{2cm}\times \sum_{\ell=1}^{j+m} \sum_{\theta = 1}^{j+m} \sum_{\mathcal{S}_{\ell, \theta}^{j,m}( \beta, \gamma)} j! m! \left(|\xi|^{2(\sigma-\sigma_1)}t\right)^{\sum_{\rho=1}^\theta \beta_{\rho}}\prod_{\rho = 1}^\theta \frac{1}{\beta_{\rho} !} \left(\frac{1}{\gamma_{1, \rho}! \gamma_{2, \rho}! }\right)^{\beta_{\rho}}\\
    &\qquad\lesssim e^{-c|\xi|^{2(\sigma-\sigma_1)}t} |\xi|^{2j(\sigma_2-\sigma_1)+2m(\sigma-2\sigma_1)}
\end{align*}
for all $(|\xi|, a, b) \in (0, \varepsilon^*] \times [0,1] \times [0,1]$, where $c$ and $c_1$ are suitable positive constants. Thus, it follows (\ref{estimate2.3.2}) and (\ref{Re2.3.2}). To prove the properties of the functions containing the term $\lambda_2^0(|\xi|,a,b)$, we use Leibniz's formula for $\lambda_2^0(|\xi|,a,b)$ in this way
\begin{align*}
    &\frac{\partial^{j+m}}{\partial a^j \partial b^m} \lambda_2^0(|\xi|,a,b) \\
    &\qquad = -\frac{|\xi|^{2\sigma_1}}{2} \sum_{h = 0}^{j} \frac{j!}{h! (j-h)!} \frac{\partial^h}{\partial a^h} \left(1+ a|\xi|^{2(\sigma_2-\sigma_1)}\right) \frac{\partial^{j+m-h}}{\partial a^{j-h} \partial b^m} (1+\Gamma_2(|\xi|,a,b))\\
    &\qquad = -\frac{|\xi|^{2\sigma_1}}{2} \bigg( \left(1+a|\xi|^{2(\sigma_2-\sigma_1)}\right) \frac{\partial^{j+m}}{\partial a^j \partial b^m} \left(1+\Gamma_2(|\xi|,a,b)\right) \\
    &\hspace{5cm} + j|\xi|^{2(\sigma_2-\sigma_1)} \frac{\partial^{j+m-1}}{\partial a^{j-1} \partial b^m} (1+\Gamma_2(|\xi|,a,b))\bigg).
\end{align*}
From this, we immediately obtain (\ref{estimate2.3.3}) and (\ref{Re2.3.3}). Then, the relations (\ref{estimate2.3.4}) and (\ref{Re2.3.4}) are proved in the same way as (\ref{estimate2.3.2}) and (\ref{Re2.3.2}). Hence, the proof of Lemma \ref{lemma2.3} is complete.
\end{proof}
By performing proof steps similar to Lemma \ref{lemma2.3}, we also obtain the following lemma.
\begin{lemma}\label{lemma2.4}
    Let $j,m \in \mathbb{N}$ and $0 \leq \sigma_1 < \sigma/2 < \sigma_2 \leq \sigma$. Then, we have the following estimate:
    \begin{align*}
\left|\frac{\partial^{j+m}}{\partial a^j \partial b^m} (G(|\xi|,a,b))^{-1}\right| &\lesssim |\xi|^{-2\sigma_1+ 2j(\sigma_2-\sigma_1)+2m(\sigma-2\sigma_1)},
    \end{align*}
    for all $ (|\xi|, a,b) \in (0, \varepsilon^*]\times [0,1] \times [0,1]$.
      Moreover, there exist some constant $C_{j,m}$ satisfying
      \begin{align*}
\frac{\partial^{j+m}}{\partial a^j \partial b^m} (G(|\xi|,a,b))^{-1}\bigg|_{(a,b)=(0,0)} &= C_{j,m} |\xi|^{-2\sigma_1+ 2j(\sigma_2-\sigma_1)+2m(\sigma-2\sigma_1)}.
    \end{align*}
\end{lemma}
From Lemmas \ref{lemma2.3} and \ref{lemma2.4}, we can prove the following proposition.
\begin{proposition}\label{pro2.1}
    Let $j, m \in \mathbb{N}$ and $0 \leq \sigma_1 < \sigma/2 < \sigma_2 \leq \sigma$. Then, we have the following estimates:
    \begin{align*}
        \left|\frac{\partial^{j+m}}{\partial a^j \partial b^m} \widehat{K_0^1}(t, |\xi|, a,b)\right| &\lesssim e^{-c|\xi|^{2\sigma_1}t} |\xi|^{2(\sigma-2\sigma_1)+ 2j(\sigma_2-\sigma_1)+ 2m(\sigma-2\sigma_1)},\\
        \left|\frac{\partial^{j+m}}{\partial a^j \partial b^m} \widehat{K_0^2}(t, |\xi|, a,b)\right| &\lesssim e^{-c|\xi|^{2(\sigma-\sigma_1)}t} |\xi|^{ 2j(\sigma_2-\sigma_1)+ 2m(\sigma-2\sigma_1)},\\
        \left|\frac{\partial^{j+m}}{\partial a^j \partial b^m} \widehat{K_1^1}(t, |\xi|, a,b)\right| &\lesssim e^{-c|\xi|^{2(\sigma-\sigma_1)}t} |\xi|^{-2\sigma_1+ 2j(\sigma_2-\sigma_1)+2m(\sigma-2\sigma_1)},\\
        \left|\frac{\partial^{j+m}}{\partial a^j \partial b^m} \widehat{K_1^2}(t, |\xi|, a,b)\right| &\lesssim e^{-c|\xi|^{2\sigma_1}t} |\xi|^{-2\sigma_1+ 2j(\sigma_2-\sigma_1)+2m(\sigma-2\sigma_1)},
    \end{align*}
    for all $(|\xi|,a,b) \in (0,\varepsilon^*] \times [0,1]\times [0,1]$ and $c$ is a suitable positive constant. Furthermore, there exist some constants $ C_{1,h,j,m}^*, C_{2, h,j,m}^*, C_{3,h,j,m}^*, C_{4,h,j,m}^* (0 \leq h \leq j+m)$ satisfying
    \begin{align*}
        &\frac{\partial^{j+m}\widehat{K_0^1}}{\partial a^j \partial b^m} (t, |\xi|, 0,0) \\
        &\qquad = e^{-|\xi|^{2\sigma_1}t} |\xi|^{2(\sigma-2\sigma_1)+ 2j(\sigma_2-\sigma_1)+ 2m(\sigma-2\sigma_1)} \sum_{h=0}^{j+m} C_{1,h,j,m}^{*}(|\xi|^{2\sigma_1} t)^h,\\
        &\frac{\partial^{j+m}\widehat{K_0^2}}{\partial a^j \partial b^m} (t, |\xi|, 0,0) \\
        &\qquad = e^{-|\xi|^{2(\sigma-\sigma_1)}t} |\xi|^{ 2j(\sigma_2-\sigma_1)+ 2m(\sigma-2\sigma_1)} \sum_{h=0}^{j+m} C_{2,h,j,m}^{*}(|\xi|^{2(\sigma-\sigma_1)}t)^h,\\
        &\frac{\partial^{j+m}\widehat{K_1^1}}{\partial a^j \partial b^m} (t, |\xi|, 0,0) \\
        &\qquad = e^{-|\xi|^{2(\sigma-\sigma_1)}t} |\xi|^{-2\sigma_1+ 2j(\sigma_2-\sigma_1)+2m(\sigma-2\sigma_1)} \sum_{h=0}^{j+m} C_{3,h,j,m}^{*}(|\xi|^{2(\sigma-\sigma_1)}t)^h,\\
        &\frac{\partial^{j+m}\widehat{K_1^2}}{\partial a^j \partial b^m} (t, |\xi|, 0,0) \\
        &\qquad = e^{-|\xi|^{2\sigma_1}t} |\xi|^{-2\sigma_1+ 2j(\sigma_2-\sigma_1)+2m(\sigma-2\sigma_1)} \sum_{h=0}^{j+m} C_{4,h,j,m}^{*}(|\xi|^{2\sigma_1}t)^h.
    \end{align*}
\end{proposition}
\begin{proof}
    By applying Leibniz's formula, we obtain the following relations:
    \begin{align*}
        &\frac{\partial^{j+m}}{\partial a^j \partial b^m} \widehat{K_0^1}(t,|\xi|,a,b) \\
        &\quad= \sum_{(0,0) \leq (j_1, m_1) \leq (j,m)} \frac{j!m!}{j_1! m_1 ! (j-j_1)! (m-m_1)!} \frac{\partial^{j_1+m_1}}{\partial a^{j_1}\partial b^{m_1}} \frac{\lambda_1^0(|\xi|,a,b)}{G(|\xi|, a,b)} \frac{\partial^{j-j_1+m-m_1}}{\partial a^{j-j_1} \partial b^{m-m_1}} e^{\lambda_2^0(|\xi|,a,b)t},\\
        &\frac{\partial^{j+m}}{\partial a^j \partial b^m} \widehat{K_0^2}(t,|\xi|,a,b) \\
        &\quad= \sum_{(0,0) \leq (j_1, m_1) \leq (j,m)} \frac{j!m!}{j_1! m_1 ! (j-j_1)! (m-m_1)!} \frac{\partial^{j_1+m_1}}{\partial a^{j_1}\partial b^{m_1}} \frac{\lambda_2^0(|\xi|,a,b)}{G(|\xi|, a,b)} \frac{\partial^{j-j_1+m-m_1}}{\partial a^{j-j_1} \partial b^{m-m_1}} e^{\lambda_1^0(|\xi|,a,b)t},\\
        &\frac{\partial^{j+m}}{\partial a^j \partial b^m} \widehat{K_1^1}(t,|\xi|,a,b) \\
        &\quad= \sum_{(0,0) \leq (j_1, m_1) \leq (j,m)} \frac{j!m!}{j_1! m_1 ! (j-j_1)! (m-m_1)!} \frac{\partial^{j_1+m_1}}{\partial a^{j_1}\partial b^{m_1}} (G(|\xi|,a,b))^{-1} \frac{\partial^{j-j_1+m-m_1}}{\partial a^{j-j_1} \partial b^{m-m_1}} e^{\lambda_1^0(|\xi|,a,b)t},\\
        &\frac{\partial^{j+m}}{\partial a^j \partial b^m} \widehat{K_1^2}(t,|\xi|,a,b) \\
        &\quad= \sum_{(0,0) \leq (j_1, m_1) \leq (j,m)} \frac{j!m!}{j_1! m_1 ! (j-j_1)! (m-m_1)!} \frac{\partial^{j_1+m_1}}{\partial a^{j_1}\partial b^{m_1}} (G(|\xi|,a,b))^{-1}\frac{\partial^{j-j_1+m-m_1}}{\partial a^{j-j_1} \partial b^{m-m_1}} e^{\lambda_2^0(|\xi|,a,b)t}.
    \end{align*}
    For the terms $\widehat{K_1^1}(t, |\xi|,a,b)$ and $\widehat{K_1^2}(t, |\xi|,a,b)$, the desired results can be concluded from Lemmas \ref{lemma2.3} and \ref{lemma2.4}. To treat the terms $\widehat{K_0^1}(t,|\xi|,a,b)$ and $\widehat{K_0^2}(t,|\xi|,a,b)$, we need to apply Leibniz's formula once again for 
    \begin{align*}
        \frac{\partial^{j_1+m_1}}{\partial a^{j_1}\partial b^{m_1}} \frac{\lambda_1^0(|\xi|,a,b)}{G(|\xi|, a,b)} &\text{ and } \frac{\partial^{j_1+m_1}}{\partial a^{j_1}\partial b^{m_1}} \frac{\lambda_2^0(|\xi|,a,b)}{G(|\xi|, a,b)}.
    \end{align*}
    Then, combining them with Lemmas \ref{lemma2.3} and \ref{lemma2.4} we may conclude Proposition \ref{pro2.1}.
\end{proof}

%....................................................................
\subsection{Estimates for high frequencies}
We are now going to proceed with the estimates for high frequencies in the following lemmas.
\begin{lemma}\label{lemma2.5}
    Let $n \geq 1$ and $0 \leq \sigma_1 < \sigma/2 < \sigma_2 \leq \sigma$. Then, the following estimate holds for all $s \geq 0$ and $t \geq 1$:
    \begin{align*}
        \| |\xi|^s \widehat{u}(t,\xi)\chi_{\rm H}(|\xi|)\|_{L^2} \lesssim e^{-ct} (\|u_0\|_{H^s}+ \|u_1\|_{H^{[s-2\sigma_2]^+}}),
    \end{align*}
    where $c$ is a suitable positive constant.
\end{lemma}
\begin{proof}
    Thanks to the representation formula (\ref{represen1}) of solutions to (\ref{Main.Eq.1}) and Parseval's formula, we obtain
    \begin{align*}
        \||\xi|^s  \widehat{u}(t,\xi)\chi_{\rm H}(|\xi|)\|_{L^2} \leq \left\|  \widehat{K_0}(t,\xi) \chi_{\rm H}(|\xi|)\right\|_{L^{\infty}} \|u_0\|_{H^s} + \left\||\xi|^{2\sigma_2} \widehat{K_1}(t,\xi) \chi_{\rm H}(|\xi|)\right\|_{L^\infty} \|u_1\|_{H^{[s-2\sigma_2]^+}}.
    \end{align*}
    From the relation (\ref{Rela1}) and the fact that
  \begin{align*}
      \left\|\widehat{K_0}(t,\xi) \chi_{\rm H}(|\xi|)\right\|_{L^\infty} &\lesssim \left(e^{-c_1|\xi|^{2(\sigma-\sigma_2)}t} +  |\xi|^{2(\sigma-2\sigma_2)}e^{-c_1|\xi|^{2\sigma_2}t}\right) \chi_{\rm H}(|\xi|) \lesssim e^{-ct},\\
      \left\||\xi|^{2\sigma_2}\widehat{K_1}(t,\xi) \chi_{\rm H}(|\xi|)\right\|_{L^\infty} &\lesssim \left(e^{-c_1 |\xi|^{2(\sigma-\sigma_2)}t} + e^{-c_1 |\xi|^{2\sigma_2}t}\right)\chi_{\rm H}(|\xi|) \lesssim e^{-ct},
  \end{align*}
where $c$ and $c_1$ are suitable positive constants, we may conclude Lemma \ref{lemma2.5}.
\end{proof}
\begin{lemma}\label{lemma2.6}
    Let $n \geq 1$, $k \in \mathbb{N}$ and $0 \leq \sigma_1 < \sigma/2 < \sigma_2 \leq \sigma$. Then, the following estimates hold for all $s \geq 0$ and $t \geq 1$:
    \begin{align*}
        \left\||\xi|^s \mathcal{A}_0^k(t,\xi) \chi_{\rm H}(|\xi|)\right\|_{L^\infty} \lesssim e^{-ct} , \,\, \left\||\xi|^s \mathcal{A}_1^k(t,\xi)\chi_{\rm H}(|\xi|)\right\|_{L^\infty} \lesssim e^{-ct}, 
    \end{align*}
    and
    \begin{align*}
        \left\||\xi|^s \mathcal{B}_0^k(t,\xi)\chi_{\rm H}(|\xi|)\right\|_{L^\infty} \lesssim e^{-ct} , \,\, \left\||\xi|^s \mathcal{B}_1^k(t,\xi)\chi_{\rm H}(|\xi|)\right\|_{L^\infty} \lesssim e^{-ct}, 
    \end{align*}
    where $c$ is suitable positive constant.
\end{lemma}
\begin{proof}
    From the definitions of $\mathcal{A}_0^k(t,\xi)$ and $\mathcal{A}_1^k(t,\xi)$, it follows that
    \begin{align*}
        &\left\||\xi|^s \mathcal{A}_0^k(t,\xi) \chi_{\rm H}(|\xi|)\right\|_{L^\infty} +  \left\||\xi|^s \mathcal{B}_0^k(t,\xi)\chi_{\rm H}(|\xi|)\right\|_{L^\infty} \\
        &\hspace{2cm} \leq  \sum_{0 \leq j+m \leq k-1}\frac{1}{j!m!} \left\||\xi|^s\chi_{\rm H}(|\xi|)\frac{\partial^{j+m} \widehat{K_0^1}}{\partial a^j \partial b^m}(t,|\xi|, 0,0)\right\|_{L^\infty}\\
        &\hspace{5cm}+ 2 \sum_{0 \leq j+m \leq k-1} \frac{1}{j!m!}\left\||\xi|^s\chi_{\rm H}(|\xi|)\frac{\partial^{j+m} \widehat{K_0^2}}{\partial a^j \partial b^m}(t,|\xi|, 0,0)\right\|_{L^\infty},\\
        &\left\||\xi|^s \mathcal{A}_1^k(t,\xi)\chi_{\rm H}(|\xi|)\right\|_{L^\infty} +  \left\||\xi|^s \mathcal{B}_1^k(t,\xi)\chi_{\rm H}(|\xi|)\right\|_{L^\infty}\\
        &\hspace{2cm}\leq  2 \sum_{0 \leq j+m \leq k-1} \frac{1}{j!m!}\left\||\xi|^s\chi_{\rm H}(|\xi|)\frac{\partial^{j+m} \widehat{K_1^1}}{\partial a^j \partial b^m}(t,|\xi|, 0,0)\right\|_{L^\infty} \\
        &\hspace{5cm}+ \sum_{0 \leq j+m \leq k-1} \frac{1}{j!m!}\left\||\xi|^s\chi_{\rm H}(|\xi|)\frac{\partial^{j+m} \widehat{K_1^2}}{\partial a^j \partial b^m}(t,|\xi|, 0,0)\right\|_{L^\infty}.
    \end{align*}
    By applying Proposition \ref{pro2.1}, we arrive at
    \begin{align*}
        &\left\||\xi|^s\frac{\partial^{j+m} \widehat{K_1^1}}{\partial a^j \partial b^m}(t,|\xi|, 0,0)\right\|_{L^\infty(|\xi| \geq \varepsilon^*/2)}\\
        &\qquad= \left\|e^{-|\xi|^{2(\sigma-\sigma_1)}t} |\xi|^{s-2\sigma_1+ 2j(\sigma_2-\sigma_1)+2m(\sigma-2\sigma_1)} \sum_{h=1}^{j+m} C_{3,h,j,m}^{*}(|\xi|^{2(\sigma-\sigma_1)}t)^h\right\|_{L^\infty(|\xi| \geq \varepsilon^*/2)}\\
        &\qquad= t ^{-\alpha_{s,j,m}} \left\|e^{-|\xi|^{2(\sigma-\sigma_1)}t} (|\xi|^{2(\sigma-\sigma_1)}t)^{\alpha_{s,j,m}} \sum_{h=1}^{j+m} C_{3,h,j,m}^{*}(|\xi|^{2(\sigma-\sigma_1)}t)^h\right\|_{L^\infty(|\xi| \geq \varepsilon^*/2)}\\
        &\qquad\lesssim t^{-\alpha_{s,j,m}} \left\|e^{-\frac{1}{2}|\xi|^{2(\sigma-\sigma_1)}t}\right\|_{L^\infty(|\xi| \geq \varepsilon^*/2)} \lesssim e^{-ct}
    \end{align*}
  for all $j, m \in \mathbb{N}$ and $t 
 \geq 1$, where $$\alpha_{s,j,m} := \frac{s-2\sigma_1+2j(\sigma_2-\sigma_1)+2m(\sigma-2\sigma_1)}{2(\sigma-\sigma_1)}.$$
   Therefore, performing some steps similar to the estimates for the terms
  $\widehat{K_0^1}$, $\widehat{K_0^2}$ and $\widehat{K_1^2}$ we may conclude Lemma \ref{lemma2.6}.
\end{proof}

\subsection{Proof of main results}
Before proving Theorems \ref{Linear_Asym} and \ref{Linear_Asym_1}, we need to recall the following important auxiliary ingredients.
\begin{lemma}\label{lemma2.7}
 Let $n \geq 1, c > 0$, $\beta > 0$ and $\alpha > -n/2$. Then, it holds
 \begin{align*}
     \left\||\xi|^{\alpha} e^{-c|\xi|^{\beta}t} \chi_{\rm L}(|\xi|)\right\|_{L^2} \lesssim (1+t)^{-\frac{n}{2\beta}-\frac{\alpha}{\beta}} \text{ for all } t >0.
 \end{align*}
 \end{lemma}
 \begin{proof}
     By changing the variable $\eta = \xi (1+t)^{\frac{1}{\beta}}$, we obtain
     \begin{align*}
         \left\||\xi|^{\alpha} e^{-c|\xi|^{\beta}t} \chi_{\rm L}(|\xi|)\right\|_{L^2}^2 &= \int_{|\xi| \leq \varepsilon^*} |\xi|^{2\alpha} e^{-2c|\xi|^{\beta}t} \chi_{\rm L}^2(|\xi|) d\xi \\
         &\leq C \int_{|\xi| \leq \varepsilon^*} |\xi|^{2\alpha} e^{-2c|\xi|^{\beta}(t+1)} \chi_{\rm L}^2(|\xi|) d\xi\\
         &\leq C(1+t)^{-\frac{n}{\beta}-\frac{2\alpha}{\beta}} \int_{\mathbb{R}^n} |\eta|^{2\alpha} e^{-2c|\eta|^{\beta}} d\eta.
     \end{align*}
     Note that the final integral converges because of $\alpha > -n/2$. This finishes the proof of Lemma \ref{lemma2.7}.
 \end{proof}

 \begin{lemma}[see \cite{IkehataTakeda2019}] \label{L^1.Lemma}
Let $s\ge 0$. Let us assume $h= h(x) \in L^1$ and $\phi=\phi(t,x)$ be a smooth function satisfying
$$ \big\||D|^s \phi(t,\cdot)\big\|_{L^2} \lesssim t^{-\alpha} \quad \text{ and }\quad  \big\||D|^{s+1} \phi(t,\cdot)\big\|_{L^2} \lesssim t^{-\alpha-\beta}, $$
for some positive constants $\alpha,\,\beta>0$. Then, it holds:
$$ \left\||D|^s \left(\phi(t,x) \ast_x h(x)- \left(\int_{\R^n}h(y)\,dy\right)\phi(t,x)\right)(t,\cdot) \right\|_{L^2}= o\big(t^{-\alpha}\big) \quad \text{ as }t \to \ity, $$
for all space dimensions $n\ge 1$.
\end{lemma}

\begin{proof}[\textbf{Proof of Theorem \ref{Linear_Asym}}] First of all, we show the estimate (\ref{UpperEs2.1}). Indeed, using Parseval's formula and Young's convolution inequality one gets
\begin{align}
   &\bigg\||\xi|^s  \widehat{u}(t,\xi) - |\xi|^s \mathcal{A}_0^k(t,\xi) \widehat{u_0}(x) - |\xi|^s\mathcal{A}_1^k(t,\xi) \widehat{u_1}(\xi)\bigg\|_{L^2}\notag\\
   &\quad\leq \left\||\xi|^s\left(\widehat{K_0}(t,\xi)- \mathcal{A}_0^k(t,\xi)\right)\chi_{\rm L}(|\xi|)\right\|_{L^2} \|u_0\|_{L^1} + \left\||\xi|^s \left(\widehat{K_1}(t,\xi)- \mathcal{A}_1^k(t,\xi)\right)\chi_{\rm L}(|\xi|)\right\|_{L^2} \|u_1\|_{L^1}\notag\\
   \vspace{0.5cm}
   &\quad\quad+ \left\||\xi|^s \widehat{u}(t,\xi) \chi_{\rm H}(|\xi|)\right\|_{L^2} + \left\||\xi|^s \mathcal{A}_0^k(t,\xi) \chi_{\rm H}(|\xi|)\widehat{u_0}(\xi)\right\|_{L^2} + \left\||\xi|^s \mathcal{A}_1^k(t,\xi) \chi_{\rm H}(|\xi|)\widehat{u_1}(\xi)\right\|_{L^2}. \label{Relation1} 
\end{align}
After applying Lemmas \ref{lemma2.5} and \ref{lemma2.6}, we gain
\begin{align}
    &\left\||\xi|^s \widehat{u}(t,\xi) \chi_{\rm H}(|\xi|)\right\|_{L^2} + \left\||\xi|^s \mathcal{A}_0^k(t,\xi) \chi_{\rm H}(|\xi|)\widehat{u_0}(\xi)\right\|_{L^2} + \left\||\xi|^s \mathcal{A}_1^k(t,\xi) \chi_{\rm H}(|\xi|)\widehat{u_1}(\xi)\right\|_{L^2}\notag\\
    &\quad\leq \left\||\xi|^s \widehat{u}(t,\xi) \chi_{\rm H}(|\xi|)\right\|_{L^2} + \left\||\xi|^s  \mathcal{A}_0^k(t,\xi)\chi_{\rm H}(|\xi|)\right\|_{L^\infty} \|u_0\|_{L^2} + \left\||\xi|^s \mathcal{A}_1^k(t,\xi)\chi_{\rm H}(|\xi|)\right\|_{L^\infty} \|u_1\|_{L^2}\notag\\
    &\quad\lesssim e^{-ct} (\|u_0\|_{H^s}+ \|u_1\|_{H^{[s-2\sigma_2]^+}}), \label{Relation2}
\end{align}
where $c$ is a suitable positive constant. 
Next, using the representations (\ref{repre1}) and (\ref{repre2}) one realizes
\begin{align*}
    &\left\||\xi|^s\left(\widehat{K_0}(t,\xi)- \mathcal{A}_0^k(t,\xi)\right)\chi_{\rm L}(|\xi|)\right\|_{L^2} \\
    &\quad\leq \left\||\xi|^s\left(\widehat{K_0^1}(t,|\xi|,1,1)-\sum_{0 \leq j+k \leq k-1}\frac{1}{j!m!}\frac{\partial^{j+m} \widehat{K_0^1}}{\partial a^j \partial b^m}(t,|\xi|, 0,0)\right)\chi_{\rm L}(|\xi|)\right\|_{L^2}\\
    &\hspace{3cm}+\left\||\xi|^s\left(\widehat{K_0^2}(t,|\xi|,1,1)-\sum_{0 \leq j+k \leq k-1}\frac{1}{j! m!}\frac{\partial^{j+m} \widehat{K_0^2}}{\partial a^j \partial b^m}(t,|\xi|, 0,0)\right)\chi_{\rm L}(|\xi|)\right\|_{L^2}
\end{align*}
and 
\vspace{0.3cm}
\begin{align*}
    &\left\||\xi|^s\left(\widehat{K_1}(t,\xi)- \mathcal{A}_1^k(t,\xi)\right)\chi_{\rm L}(|\xi|)\right\|_{L^2} \\
    &\quad\leq \left\||\xi|^s\left(\widehat{K_1^1}(t,|\xi|,1,1)-\sum_{0 \leq j+k \leq k-1}\frac{1}{j!m!}\frac{\partial^{j+m} \widehat{K_1^1}}{\partial a^j \partial b^m}(t,|\xi|, 0,0)\right)\chi_{\rm L}(|\xi|)\right\|_{L^2}\\
    &\hspace{3cm}+\left\||\xi|^s\left(\widehat{K_1^2}(t,|\xi|,1,1)-\sum_{0 \leq j+k \leq k-1}\frac{1}{j! m!}\frac{\partial^{j+m} \widehat{K_1^2}}{\partial a^j \partial b^m}(t,|\xi|, 0,0)\right)\chi_{\rm L}(|\xi|)\right\|_{L^2}.
\end{align*}
By Taylor's theorem, we arrive at
\begin{align*}
    \widehat{K_0^1}(t,|\xi|,1,1)-\sum_{0 \leq j+k \leq k-1}\frac{1}{j!m!}\frac{\partial^{j+m} \widehat{K_0^1}}{\partial a^j \partial b^m}(t,|\xi|, 0,0) &= \sum_{j+m =k } \frac{1}{j!m!} \frac{\partial^{j+m} \widehat{K_0^1}}{\partial a^j \partial b^m}(t,|\xi|, a_1,b_1),\\
    \widehat{K_0^2}(t,|\xi|,1,1)-\sum_{0 \leq j+k \leq k-1}\frac{1}{j!m!}\frac{\partial^{j+m} \widehat{K_0^2}}{\partial a^j \partial b^m}(t,|\xi|, 0,0) &= \sum_{j+m =k } \frac{1}{j!m!} \frac{\partial^{j+m} \widehat{K_0^2}}{\partial a^j \partial b^m}(t,|\xi|, a_2,b_2),\\
    \widehat{K_1^1}(t,|\xi|,1,1)-\sum_{0 \leq j+k \leq k-1}\frac{1}{j!m!}\frac{\partial^{j+m} \widehat{K_1^1}}{\partial a^j \partial b^m}(t,|\xi|, 0,0) &= \sum_{j+m =k } \frac{1}{j!m!} \frac{\partial^{j+m} \widehat{K_1^1}}{\partial a^j \partial b^m}(t,|\xi|, a_3,b_3),\\
    \widehat{K_1^2}(t,|\xi|,1,1)-\sum_{0 \leq j+k \leq k-1}\frac{1}{j!m!}\frac{\partial^{j+m} \widehat{K_1^2}}{\partial a^j \partial b^m}(t,|\xi|, 0,0) &= \sum_{j+m =k } \frac{1}{j!m!} \frac{\partial^{j+m} \widehat{K_1^2}}{\partial a^j \partial b^m}(t,|\xi|, a_4,b_4),
\end{align*}
with some $(a_h,b_h) \in [0,1] \times [0,1]$. To demonstrate the desired estimates, we will apply Proposition \ref{pro2.1} and Lemma \ref{lemma2.7} in each specific case as follows:
\begin{itemize}
 [leftmargin=*]
    \item If $\alpha = s+2(\sigma-2\sigma_1)+2k\delta$ and $\beta =2\sigma_1$, then 
    \begin{align*}
       \left\||\xi|^s \chi_{\rm L}(|\xi|) \sum_{j+m =k } \frac{1}{j!m!} \frac{\partial^{j+m} \widehat{K_0^1}}{\partial a^j \partial b^m}(t,|\xi|, a_1,b_1) \right\|_{L^2} &\lesssim \left\|e^{-c|\xi|^{2\sigma_1}t} |\xi|^{s+2(\sigma-2\sigma_1)+2k\delta}\chi_{\rm L}(|\xi|)\right\|_{L^2}\\
       &\lesssim (1+t)^{-\frac{n}{4\sigma_1}-\frac{\sigma-\sigma_1}{\sigma_1}+1-\frac{s}{2\sigma_1}-k \frac{\delta}{\sigma_1}}.
    \end{align*}
    \item If $\alpha = s+ 2k\delta$ and $\beta = 2(\sigma-\sigma_1)$, then
    \begin{align*}
        \left\||\xi|^s \chi_{\rm L}(|\xi|) \sum_{j+m =k } \frac{1}{j!m!} \frac{\partial^{j+m} \widehat{K_0^2}}{\partial a^j \partial b^m}(t,|\xi|, a_2,b_2)\right\|_{L^2} &\lesssim \left\|e^{-c|\xi|^{2(\sigma-\sigma_1)}t} |\xi|^{s+2k\delta} \chi_{\rm L}(|\xi|)\right\|_{L^2}\\
        &\lesssim (1+t)^{-\frac{n}{4(\sigma-\sigma_1)}-\frac{s}{2(\sigma-\sigma_1)}-k \frac{\delta}{\sigma-\sigma_1}}.
    \end{align*}
    \item  If $\alpha = s-2\sigma_1+2k\delta$ and $\beta = 2(\sigma-\sigma_1)$, then
    \begin{align*}
        \left\||\xi|^s\chi_{\rm L}(|\xi|)\sum_{j+m =k } \frac{1}{j!m!} \frac{\partial^{j+m} \widehat{K_1^1}}{\partial a^j \partial b^m}(t,|\xi|, a_3,b_3)\right\|_{L^2} &\lesssim \left\|e^{-c|\xi|^{2(\sigma-\sigma_1)}t} |\xi|^{s-2\sigma_1+2k\delta}\chi_{\rm L}(|\xi|)\right\|_{L^2}\\
        &\lesssim (1+t)^{-\frac{n}{4(\sigma-\sigma_1)}+\frac{\sigma_1}{\sigma-\sigma_1}-\frac{s}{2(\sigma-\sigma_1)}-k \frac{\delta}{\sigma-\sigma_1}},
    \end{align*}
    where we note that the condition $n > 4\sigma_1$ implies $\alpha > -n/2$ for all $k \geq 0$ and $s \geq 0$.
    \item If $\alpha = s-2\sigma_1+2k\delta$ and $\beta = 2\sigma_1$, then
    \begin{align*}
        \left\||\xi|^s\chi_{\rm L}(|\xi|)\sum_{j+m =k } \frac{1}{j!m!} \frac{\partial^{j+m} \widehat{K_1^2}}{\partial a^j \partial b^m}(t,|\xi|, a_4,b_4)\right\|_{L^2} &\lesssim \left\|e^{-c|\xi|^{2\sigma_1}t}|\xi|^{s-2\sigma_1+2k\delta}\chi_{\rm L}(|\xi|)\right\|_{L^2}\\
        &\lesssim (1+t)^{-\frac{n}{4\sigma_1}+1-\frac{s}{2\sigma_1}-k\frac{\delta}{\sigma_1}}.
    \end{align*}
\end{itemize}
Due to condition $n > 4\sigma_1$, we have the relation
\begin{align*}
    -\frac{n}{4(\sigma-\sigma_1)}+\frac{\sigma_1}{\sigma-\sigma_1}-\frac{s}{2(\sigma-\sigma_1)}-k \frac{\delta}{\sigma-\sigma_1} > -\frac{n}{4\sigma_1}+1-\frac{s}{2\sigma_1}-k \frac{\delta}{\sigma_1},
\end{align*}
for all $s \geq 0$ and $k \geq 0$. Combining this with (\ref{Relation1}) and (\ref{Relation2}), we conclude the estimate (\ref{UpperEs2.1}).\\
 To prove the estimate (\ref{LowerEs2.1}), we can proceed as follows:
    \begin{align}
        &\bigg\|  |\xi|^s\widehat{u}(t,\xi) - |\xi|^s \mathcal{A}_0^k(t,\xi)\widehat{u_0}(\xi) - |\xi|^s \mathcal{A}_1^k(t,\xi) \widehat{u_1}(\xi)\bigg\|_{L^2}\notag\\
        &\quad \geq \left\||\xi|^s\sum_{j+m = k} \frac{1}{j!m!}\frac{\partial^{j+m} \widehat{K_1^1}}{\partial a^j \partial b^m} (t,|\xi|, 0 ,0) \widehat{u_1}(\xi)\right\|_{L^2}\notag\\
        &\qquad - \left\||\xi|^s\sum_{j+m = k} \frac{1}{j! m!}\frac{\partial^{j+m} \widehat{K_1^2}}{\partial a^j \partial b^m} (t,|\xi|, 0 ,0) \widehat{u_1}(\xi)\right\|_{L^2} \notag\\
        &\qquad- \left\||\xi|^s\sum_{j+m = k} \frac{1}{j!m!}\left(\frac{\partial^{j+m} \widehat{K_0^1}}{\partial a^j \partial b^m} (t,|\xi|, 0 ,0) - \frac{\partial^{j+m} \widehat{K_0^2}}{\partial a^j \partial b^m} (t,|\xi|, 0 ,0)\right)\widehat{u_0}(\xi)\right\|_{L^2}\notag\\
        &\qquad- \bigg\||\xi|^s  \widehat{u}(t,\xi) - |\xi|^s\mathcal{A}_0^{k+1}(t,\xi)\widehat{u_0}(\xi) - |\xi|^s \mathcal{A}_1^{k+1}(t,\xi) \widehat{u_1}(\xi)\bigg\|_{L^2}\notag\\
        &\qquad =: \mathcal{J}_1(t) - \mathcal{J}_2(t)-\mathcal{J}_3(t)-\mathcal{J}_4(t).\label{Main.Re4.1}
    \end{align}
    Using the estimate (\ref{UpperEs2.1}), we obtain 
    \begin{align}
        \mathcal{J}_4(t) &\lesssim (1+t)^{-\frac{n}{4(\sigma-\sigma_1)}-\frac{s}{2(\sigma-\sigma_1)}+\frac{\sigma_1}{\sigma-\sigma_1}-(k+1)\frac{\delta}{\sigma-\sigma_1}} \|(u_0, u_1)\|_{\mathcal{D}_s}\notag\\
        &= o\left( t^{-\frac{n}{4(\sigma-\sigma_1)}-\frac{s}{2(\sigma-\sigma_1)}+\frac{\sigma_1}{\sigma-\sigma_1}-k\frac{\delta}{\sigma-\sigma_1}}\right),\label{Re4.1.0}
    \end{align}
    as $t \to \infty$. Again, applying Proposition \ref{pro2.1} linked to Young's convolution inequality one has
    \begin{align}
        \mathcal{J}_3(t) = &\left\||\xi|^s\sum_{j+m = k} \frac{1}{j!m!} \left(\frac{\partial^{j+m} \widehat{K_0^1}}{\partial a^j \partial b^m} (t,|\xi|, 0 ,0) - \frac{\partial^{j+m} \widehat{K_0^2}}{\partial a^j \partial b^m} (t,|\xi|, 0 ,0)\right)\widehat{u_0}(\xi)\right\|_{L^2}\notag\\
        &\qquad\lesssim \left(\left\|e^{-c|\xi|^{2\sigma_1}t} |\xi|^{s+2(\sigma-2\sigma_1)+2k\delta}\right\|_{L^2} +\left\|e^{-c|\xi|^{2(\sigma-\sigma_1)}t} |\xi|^{s+2k\delta}\right\|_{L^2} \right)\|u_0\|_{L^1}\notag\\
        &\qquad\lesssim t^{\max\{-\frac{n}{4\sigma_1}-\frac{s}{2\sigma_1}-\frac{\sigma-2\sigma_1}{\sigma_1} -k \frac{\delta}{\sigma_1}, \,-\frac{n}{4(\sigma-\sigma_1)}-\frac{s}{2(\sigma-\sigma_1)}-k \frac{\delta}{\sigma-\sigma_1}\}} \|u_0\|_{L^1}\notag\\
        &\qquad = o\left(t^{-\frac{n}{4(\sigma-\sigma_1)}-\frac{s}{2(\sigma-\sigma_1)}+\frac{\sigma_1}{\sigma-\sigma_1}-k\frac{\delta}{\sigma-\sigma_1}}\right) \label{Re4.1.1},
    \end{align}
    as $t \to \infty$, where $c$ is a suitable positive constant. Similarly, we can also estimate that
    \begin{align}
         \mathcal{J}_2(t) = \left\||\xi|^s\sum_{j+m = k} \frac{1}{j!m!}\frac{\partial^{j+m} \widehat{K_1^2}}{\partial a^j \partial b^m} (t,|\xi|, 0 ,0) \widehat{u_1}(\xi)\right\|_{L^2} &\lesssim t^{-\frac{n}{4\sigma_1}-\frac{s}{2\sigma_1} +1-k \frac{\delta}{\sigma_1}} \|u_1\|_{L^1} \notag\\
        &= o\left(t^{-\frac{n}{4(\sigma-\sigma_1)}-\frac{s}{2(\sigma-\sigma_1)}+\frac{\sigma_1}{\sigma-\sigma_1}-k\frac{\delta}{\sigma-\sigma_1}}\right) \label{Re4.1.2}.
    \end{align}
    Finally, for the term $\mathcal{J}_1(t)$ with condition $P_1 \neq 0$, one sees
    \begin{align}
    \mathcal{J}_1(t) = &\left\||\xi|^s\sum_{j+m = k} \frac{1}{j!m!}\frac{\partial^{j+m} \widehat{K_1^1}}{\partial a^j \partial b^m} (t,|\xi|, 0 ,0) \widehat{u_1}(\xi)\right\|_{L^2} \notag\\
        &\qquad \geq \left|P_1 \right| \left\||\xi|^s\sum_{j+m = k} \frac{1}{j!m!}\frac{\partial^{j+m} \widehat{K_1^1}}{\partial a^j \partial b^m} (t,|\xi|, 0 ,0)\right\|_{L^2}\notag\\
        &\hspace{3cm}- \left\||\xi|^s\sum_{j+m = k}  \frac{1}{j!m!}\frac{\partial^{j+m}\widehat{K_1^1}}{\partial a^j \partial b^m} (t,|\xi|, 0 ,0) \left(\widehat{u_1}(\xi) - P_1\right)\right\|_{L^2}. \label{Re4.1.3}
    \end{align}
    Using again Proposition \ref{pro2.1}, we get
    \begin{align*}
        &\left\||\xi|^s\sum_{j+m = k} \frac{1}{j!m!}\frac{\partial^{j+m} \widehat{K_1^1}}{\partial a^j \partial b^m} (t,|\xi|, 0 ,0)\right\|_{L^2}\\
        &\qquad= \left\|\sum_{j+m = k} \frac{1}{j!m!}e^{-|\xi|^{2(\sigma-\sigma_1)}t} |\xi|^{s-2\sigma_1+ 2j(\sigma_2-\sigma_1)+2m(\sigma-2\sigma_1)} \sum_{h=0}^{j+m} C_{3,h,j,m}^{*}(|\xi|^{2(\sigma-\sigma_1)}t)^h\right\|_{L^2}.
    \end{align*}
    Without loss of generality, we assume that $\sigma_2 +\sigma_1 < \sigma $, that is, $\delta = \sigma_2-\sigma_1$. For this reason, it entails
   
        \begin{align}
            &\left\||\xi|^s\sum_{j+m = k} \frac{1}{j!m!} \frac{\partial^{j+m} \widehat{K_1^1}}{\partial a^j \partial b^m} (t,|\xi|, 0 ,0)\right\|_{L^2}\notag\\
            &\quad \geq \frac{1}{k!}\left\|e^{-|\xi|^{2(\sigma-\sigma_1)}t} |\xi|^{s-2\sigma_1+ 2k\delta} \sum_{h=0}^{k} C_{3,h,k, 0}^{*}(|\xi|^{2(\sigma-\sigma_1)}t)^h\right\|_{L^2} \notag\\
            &\qquad - \sum_{j+m = k, m \neq 0} \frac{1}{j!m!}\left\|e^{-|\xi|^{2(\sigma-\sigma_1)}t} |\xi|^{s-2\sigma_1+ 2j(\sigma_2-\sigma_1)+2m(\sigma-2\sigma_1)} \sum_{h=0}^{j+m} C_{3,h,j,m}^{*}(|\xi|^{2(\sigma-\sigma_1)}t)^h\right\|_{L^2}\notag\\
            &\quad \gtrsim t^{-\frac{n}{4(\sigma-\sigma_1)}-\frac{s}{2(\sigma-\sigma_1)}+\frac{\sigma_1}{\sigma-\sigma_1}-k\frac{\delta}{\sigma-\sigma_1}} - o\left(t^{-\frac{n}{4(\sigma-\sigma_1)}-\frac{s}{2(\sigma-\sigma_1)}+\frac{\sigma_1}{\sigma-\sigma_1}-k\frac{\delta}{\sigma-\sigma_1}}\right), \label{Re4.1.4}
        \end{align}
    where we have utilized the estimate
    \begin{align*}
        &\left\|e^{-|\xi|^{2(\sigma-\sigma_1)}t} |\xi|^{s-2\sigma_1+ 2j(\sigma_2-\sigma_1)+2m(\sigma-2\sigma_1)} \sum_{h=0}^{j+m} C_{3,h,j,m}^{*}(|\xi|^{2(\sigma-\sigma_1)}t)^h\right\|_{L^2}\\
        &\qquad = C_{h,j,m} t^{-\frac{n}{4(\sigma-\sigma_1)}-\frac{s}{2(\sigma-\sigma_1)}+\frac{\sigma_1}{\sigma-\sigma_1}-j \frac{\sigma_2-\sigma_1}{\sigma-\sigma_1}- m \frac{\sigma-2\sigma_1}{\sigma-\sigma_1}}
    \end{align*}
    for all $s \geq 0$ and $j,m \in \mathbb{N}$. From this, applying Lemma \ref{L^1.Lemma} we derive
    \begin{align*}
        &\left\||\xi|^s\sum_{j+m = k} \frac{1}{j!m!}\frac{\partial^{j+m} \widehat{K_1^1}}{\partial a^j \partial b^m} (t,|\xi|, 0 ,0) \left(\widehat{u_1}(\xi) - P_1\right)\right\|_{L^2} \\
        &\qquad= 
        \bigg\||D|^s \mathfrak{F}_{\xi\to x}^{-1}\bigg(\sum_{j+m = k} \frac{1}{j!m!}\frac{\partial^{j+m} \widehat{K_1^1}}{\partial a^j \partial b^m} (t,|\xi|, 0 ,0)\bigg) \ast_x u_1(x)\\
        &\hspace{4cm}- P_1 |D|^s \mathfrak{F}_{\xi\to x}^{-1}\bigg(\sum_{j+m = k} \frac{1}{j!m!}\frac{\partial^{j+m} \widehat{K_1^1}}{\partial a^j \partial b^m} (t,|\xi|, 0 ,0)\bigg) \bigg\|_{L^2}
        \\
        &\qquad= o\left(t^{-\frac{n}{4(\sigma-\sigma_1)}-\frac{s}{2(\sigma-\sigma_1)}+\frac{\sigma_1}{\sigma-\sigma_1}-k\frac{\delta}{\sigma-\sigma_1}}\right),
    \end{align*}
    as $t \to \infty$.
    Combining this with the relations (\ref{Re4.1.3}) and (\ref{Re4.1.4}), we have
    \begin{align}
        \mathcal{J}_1(t) \gtrsim t^{-\frac{n}{4(\sigma-\sigma_1)}-\frac{s}{2(\sigma-\sigma_1)}+\frac{\sigma_1}{\sigma-\sigma_1}-k\frac{\delta}{\sigma-\sigma_1}}, \label{Re4.1.5}
    \end{align}
    for $t \geq 1$. From the relations (\ref{Main.Re4.1})-(\ref{Re4.1.2}) and (\ref{Re4.1.5}), we may conclude the estimate (\ref{LowerEs2.1}). In summary, Theorem \ref{Linear_Asym} has been proved.
\end{proof}

%.........................................
\begin{proof}[\textbf{Proof of Theorem \ref{Linear_Asym_1}}]
In the special case $\sigma_1 = 0$, let us express the characteristic root $\lambda_2(|\xi|)$ as follows:
\begin{align*}
   \lambda_{2}(|\xi|) = \frac{1}{2}\left(-1 - |\xi|^{2\sigma_2} - \sqrt{(1 + |\xi|^{2\sigma_2})^{2} - 4 |\xi|^{2\sigma}}\right).
\end{align*}
and
\begin{align*}
    \lambda_1(|\xi|)-\lambda_2(|\xi|) = \sqrt{(1 + |\xi|^{2\sigma_2})^{2} - 4 |\xi|^{2\sigma}}.
\end{align*}
Thus, it immediately yields
\begin{align*}
    \left\|\widehat{K_0^1}(t,\xi)\right\|_{L^\infty} \lesssim e^{-ct} &\text{ and } \left\|\widehat{K_1^2}(t,\xi)\right\|_{L^\infty} \lesssim e^{-ct}
\end{align*}
for all $t \geq 1$, where $c$ is a suitable positive constant. On the other hand, the fact is that
\begin{align*}
    &\bigg\||\xi|^s  \widehat{u}(t,\xi) - |\xi|^s \mathcal{B}_0^k(t,\xi) \widehat{u_0}(\xi) - |\xi|^s\mathcal{B}_1^k(t,\xi) \widehat{u_1}(\xi)\bigg\|_{L^2}\\
    &\quad\leq \left\||\xi|^s\left(\widehat{K_0^2}(t,\xi)- \mathcal{B}_0^k(t,\xi)\right)\chi_{\rm L}(|\xi|)\right\|_{L^2} \|u_0\|_{L^1} + \left\||\xi|^s \left(\widehat{K_1^1}(t,\xi)- \mathcal{B}_1^k(t,\xi)\right)\chi_{\rm L}(|\xi|)\right\|_{L^2} \|u_1\|_{L^1}\\
    &\quad\quad+  \left\|\widehat{K_0^1}(t,\xi) \chi_{\rm L}(|\xi|)\right\|_{L^\infty} \|u_0\|_{L^2} + \left\|\widehat{K_1^2}(t,\xi) \chi_{\rm L}(|\xi|)\right\|_{L^\infty} \|u_1\|_{L^2}\\
    &\quad\quad+ \left\||\xi|^s \widehat{u}(t,\xi) \chi_{\rm H}(|\xi|)\right\|_{L^2} + \left\||\xi|^s \mathcal{B}_0^k(t,\xi)\chi_{\rm H}(|\xi|)\right\|_{L^\infty} \|u_0\|_{L^2} + \left\||\xi|^s \mathcal{B}_1^k(t,\xi)\chi_{\rm H}(|\xi|)\right\|_{L^\infty} \|u_1\|_{L^2}.
\end{align*}
At this point, the terms
\begin{align*}
    \left\||\xi|^s\left(\widehat{K_0^2}(t,\xi)- \mathcal{B}_0^k(t,\xi)\right)\chi_{\rm L}(|\xi|)\right\|_{L^2} \text{ and } \left\||\xi|^s\left(\widehat{K_1^1}(t,\xi)- \mathcal{B}_1^k(t,\xi)\right)\chi_{\rm L}(|\xi|)\right\|_{L^2}
\end{align*}
have been estimated as in the proof of Theorem \ref{Linear_Asym} by the replacement of $\delta = \sigma_2$ if $\sigma_1 = 0$. From Lemma \ref{lemma2.5} and Lemma \ref{lemma2.6}, we may immediately conclude
\begin{align*}
    &\left\||\xi|^s \widehat{u}(t,\xi) \chi_{\rm H}(|\xi|)\right\|_{L^2} + \left\||\xi|^s \mathcal{B}_0^k(t,\xi)\chi_{\rm H}(|\xi|)\right\|_{L^\infty} \|u_0\|_{L^2} + \left\||\xi|^s \mathcal{B}_1^k(t,\xi) \chi_{\rm H}(|\xi|)\right\|_{L^\infty} \|u_1\|_{L^2} \\
    &\hspace{4cm} \lesssim e^{-ct} (\|u_0\|_{H^s} + \|u_1\|_{H^{[s-2\sigma_2]^+}}).
\end{align*}
Hence, we obtain estimate (\ref{UpperEs2.2}). By performing proof steps similar to the proof of Theorem \ref{Linear_Asym}, we obtain (\ref{LowerEs2.2}).
Therefore, the proof of Theorem \ref{Linear_Asym_1} has been completed.
\end{proof}
%.................................................

%====================================================================
\section*{Acknowledgements}
This research is funded by Vietnam National Foundation for Science and Technology Development (NAFOSTED) under grant number 101.02-2023.12. The authors wish to thank Hiroshi Takeda (Fukuoka University) for giving them helpful advice to improve this paper, in particular, to establish the lower bound estimates \eqref{LowerEs2.1} and \eqref{LowerEs2.2} in Theorems \ref{Linear_Asym} and \ref{Linear_Asym_1}, respectively.

%=================================================================================={References}


\begin{thebibliography}{00}
 \bibliographystyle{plain}
 
\bibitem{ChenDAbbiccoGirardi2022} W. Chen, M. D'Abbicco, G. Girardi, Global small data solutions for semilinear waves with two dissipative terms, \textit{Ann. Mat. Pura. Appl.}, \textbf{201} (2022), 529-560.

\bibitem{Constantine1996} G. M. Constantine, T. H. Savits, A Multivariate Fa\`a di Bruno formula with applications, \textit{Trans. Amers. Math. Soc.}, \textbf{348} (1996) 503-520.
 
\bibitem{DAbbiccoEbert2014} M. D'Abbicco, M. R. Ebert, Diffusion phenomena for the wave equation with structural damping in the $L^p-L^q$ framework, \textit{J. Differential Equations}, \textbf{256} (2014) 2307-2336. 

\bibitem{DAbbiccoEbert2017} M. D'Abbicco, M. R. Ebert, A new phenomenon in the critical exponent for structurally damped semi-linear evolution equations, \textit{Nonlinear Anal.}, \textbf{149} (2017) 1-40.

\bibitem{DAbbiccoEbert2021}  M. D'Abbicco, M. R. Ebert, $L^p - L^q$ estimates for a parameter-dependent multiplier with oscillatory and diffusive components, \textit{J. Math. Anal. Appl.}, \textbf{504} (2021) 125393.

\bibitem{DAbbiccoEbert2022} M. D'Abbicco, M. R. Ebert, The critical exponent for semilinear $\sigma$-evolution equations with a strong non-effective damping, \textit{Nonlinear Anal.}, \textbf{215} (2022) 112637.

\bibitem{DabbiccoReissig2014} M. D'Abbicco, M. Reissig, Semilinear structural damped waves, \textit{Math. Methods Appl. Sci.}, \textbf{37} (2014), 1570-1592.

\bibitem{DaoDuongNguyen2024} T. A. Dao, D. V. Duong, D. A. Nguyen, On asymptotic properties of solutions to $\sigma$-evolution equations with general double damping, \textit{J. Math. Anal. Appl.}, \textbf{536} (2024) 128246.

\bibitem{DaoMichihisa2020} T. A. Dao, H. Michihisa, Study of semi-linear $\sigma$-evolution equations with frictional and visco-elastic damping, \textit{Commun. Pure Appl. Anal.}, \textbf{19} (2020) 1581-1608.

\bibitem{DaoReissig2019_1} T. A. Dao, M. Reissig, An application of $L^1$ estimates for oscillating integrals to parabolic like semi-linear structurally damped $\sigma$-evolution models, \textit{J. Math. Anal. Appl.} \textbf{476} (2019) 426-463.

\bibitem{DaoReissig2019_2} T. A. Dao, M. Reissig, $L^1$ estimates for oscillating integrals and their applications to semi-linear models with $\sigma$- evolution like structural damping, \textit{Discrete Contin. Dynam. Systems}, \textbf{39} (2019), 5431-5463.

\bibitem{DuongKainaneReissig2015} P. T. Duong, M. M. Kainane, M. Reissig, Global existence for semi-linear structurally damped $\sigma$-evolution models, \textit{J. Math. Anal. Appl.}, \textbf{431} (2015), 569-596.

\bibitem{Ikehata2014} R. Ikehata, Asymptotic profiles for wave equations with strong damping, \textit{J. Differential Equations}, \textbf{257} (2014) 2159-2177.

\bibitem{Ikehata2021} R. Ikehata, A note on optimal $L^2$- estimates of solutions to some strongly damped $\sigma$-evolution equations, \textit{Asymptotic Anal.}, \textbf{121} (2021) 59-74. 

\bibitem{IkehataMichihisa2019} R. Ikehata, H. Michihisa, Moment conditions and lower bounds in expanding solutions of wave equations with double damping terms, \textit{Asymptotic Anal.}, \textbf{98} (2019) 19-36.

\bibitem{Ikehata2012} R. Ikehata, M. Natsume, Energy decay estimates for wave equations with a fractional damping, \textit{Differ. Integr. Equ.}, \textbf{25}(2012), 939-956.

\bibitem{IkehataSawada2016} R. Ikehata, A. Sawada, Asymptotic profile of solutions for wave equations with frictional and viscoelastic damping terms, \textit{ Asymptotic Anal.}, \textbf{98} (2016) 59-77.

\bibitem{IkehataTakeda} R. Ikehata, H. Takeda, Critical exponent for nonlinear wave equations with frictional and viscoelastic damping terms, \textit{Nonlinear Anal.}, \textbf{148} (2017) 228-253.

\bibitem{IkehataTakeda2019} R. Ikehata, H. Takeda, Asymptotic profiles of solutions for structural damped wave equations, \textit{J. Dynam. Differential Equations}, \textbf{31}(2019) 537-571.

\bibitem{KawakamiTakeda2016} T. Kawakami, H. Takeda, Higher order asymptotic expansions to the solutions for a nonlinear damped wave equation, \textit{Nonlinear Differ. Equ. Appl.}, (2016) 23:54. 

\bibitem{Matsumura1976} A. Matsumura, On the asymptotic behavior of solutions of semi-linear wave equations, \textit{Publ. Res. Inst. Math. Sci.}, \textbf{12} (1976), 169-189. 
 
\bibitem{Michihisa2021} H. Michihisa, $L^2$ asymptotic profiles of solutions to linear damped wave equations, \textit{J. Differential Equations}, \textbf{296} (2021) 573-592.

\bibitem{NarazakiReissig2013} T. Narazaki, M. Reissig, $L^1$ estimates for oscillating integrals related to structural damped wave models, in: M. Cicognani, F. Colombini, D. Del Santo (Eds.), Studies in Phase Space Analysis with Applications to PDEs, in: Progr. Nonlinear Differential Equations Appl., Birkh\"auser, (2013), 215-258.

\bibitem{Shibata2000} Y. Shibata, On the rate of decay of solutions to linear viscoelastic equation, \textit{Math. Meth. Appl. Sci.}, \textbf{23} (2000) 203-226.

\bibitem{Takeda2015} H. Takeda, Higher-order expansion of solutions for a damped wave equation, \textit{Asymptotic Anal.}, \textbf{94}(2015) 1-31.

\bibitem{TodorovaYordanov2001} G. Todorova, B. Yordanov, Critical exponent for a nonlinear wave equation with damping, \textit{J. Differential Equations}, \textbf{174} (2001) 464-489.

\bibitem{Ueda2011} Y. Ueda, T. Nakamura, S. Kawashima, Energy method in the partial Fourier space and application to stability problems in the half space, \textit{J. Differential Equations}, \textbf{250} (2011) 1169-1199.

\end{thebibliography}
\end{document}